\documentclass[11pt]{amsart}
\usepackage[T1]{fontenc}
\usepackage[utf8]{inputenc}
\usepackage[main=english]{babel}
\usepackage{geometry}
\usepackage[dvipsnames]{xcolor}
\usepackage{amsmath}
\usepackage{amsthm}
\usepackage{amssymb}
\usepackage{mathrsfs}
\usepackage{stackrel}
\usepackage{array}
\usepackage[all]{xy}
\usepackage[unicode=true, pdfusetitle, bookmarks=true, bookmarksnumbered=false, bookmarksopen=false, breaklinks=false, backref=false, colorlinks=true]
{hyperref}
\hypersetup{pdfborderstyle=,linkcolor=MidnightBlue, citecolor=PineGreen, urlcolor=PineGreen}
\usepackage{csquotes}
\usepackage[backend=biber,date=year, isbn=false, doi=false, url=false, eprint=true, style=alphabetic, maxnames=5]{biblatex}
\DeclareNameAlias{default}{given-family}

\DeclareFieldFormat[article,book,incollection,misc]{title}{%
	\iffieldundef{doi}
	{\mkbibemph{#1}}
	{\href{https://doi.org/\thefield{doi}}{\mkbibemph{#1}}}}

\renewbibmacro{in:}{}
\DeclareFieldFormat{pages}{#1}
\renewbibmacro*{note+pages}{%
	\printfield{pages}%
	\setunit{\addcomma\space}%
	\printfield{note}%
	\newunit}

\usepackage{caption}
\usepackage{booktabs}
\usepackage{makecell} 
\usepackage{enumitem}

\theoremstyle{plain}
\newtheorem{thm}{\protect\theoremname}[section]
\theoremstyle{definition}
\newtheorem{defn}[thm]{\protect\definitionname}
\theoremstyle{remark}
\newtheorem{rem}[thm]{\protect\remarkname}
\theoremstyle{plain}
\newtheorem{lem}[thm]{\protect\lemmaname}
\theoremstyle{plain}
\newtheorem{prop}[thm]{\protect\propositionname}
\theoremstyle{plain}
\newtheorem{cor}[thm]{\protect\corollaryname}
\theoremstyle{definition}
\newtheorem*{problem*}{\protect\problemname}

\providecommand{\corollaryname}{Corollary}
\providecommand{\definitionname}{Definition}
\providecommand{\lemmaname}{Lemma}
\providecommand{\problemname}{Problem}
\providecommand{\propositionname}{Proposition}
\providecommand{\remarkname}{Remark}
\providecommand{\theoremname}{Theorem}

\DeclareMathOperator{\codim}{codim}
\DeclareMathOperator{\Vol}{Vol}
\DeclareMathOperator{\MV}{MV}
\DeclareMathOperator{\cay}{cay}
\DeclareMathOperator{\ev}{ev}
\DeclareMathOperator{\rk}{rk}
\newcommand{\Det}{\mathrm{Det}}
\newcommand{\Mat}{\mathrm{Mat}}

\makeatletter
\long\def\gobblefivebm#1#2#3#4#5{}
\newcommand{\sectionnobookmark}[1]{%
	\begingroup
	\let\Hy@writebookmark\gobblefivebm   % hyperref's bookmark writer, disabled locally
	\section{#1}%
	\endgroup
}
\makeatother

\title[Components of discriminants for polynomial systems] {Components of discriminants for systems of equations and irreducibility of determinants}
\author{Vladislav Pokidkin}
\address{HSE University, Moscow, Russia}
\email{pokidkin.vladislav@gmail.com}
\thanks{The study has been funded within the framework of the HSE University Basic Research Program (HSE-BR-2025-061) and by the Foundation for the Advancement of Theoretical Physics and Mathematics \enquote{BASIS}}

\begin{document}
	\begin{abstract}
		The discriminant of a multivariate polynomial with indeterminate coefficients is not necessarily a hypersurface, and the problem of characterizing its codimension remained open for an extended period. We solve this problem for the discriminants of systems of polynomials with indeterminate coefficients in which the number of equations equals the number of unknowns (square polynomial systems). Such discriminants may have several components of different dimensions.	In the space of square matrices, we characterize row-generated subspaces on which the determinant is an irreducible polynomial. This allows us to prove Esterov's conjecture that the discriminant of a nonlinear irreducible square polynomial system is an irreducible hypersurface. Building on this result, we enumerate all components and determine their dimensions and degrees for each of the three conventional formalizations of the notion of a discriminant in this setting (mixed, Cayley, and $A$-discriminants), for square systems and for those overdetermined systems that do not reduce to underdetermined ones. Both the proof of the Esterov conjecture and the descriptions of the three types of discriminants are based on the theory of polymatroids.
	\end{abstract}
	
	\maketitle
	
	\section{Introduction}
	
	Let $T\simeq(\mathbb{C}^{\times})^{n}$ be an algebraic torus with its character lattice $M\simeq\mathbb{Z}^{n}$. A finite set of monomials $A\subset M$ spans a vector space of Laurent polynomials, denoted by $\mathbb{C}_{A}$. For a tuple of finite sets $\mathscr{A}=(A_{1},\ldots,A_{k})\subset M$, much attention has been paid, since the work of Gelfand, Kapranov, and Zelevinsky \cite{gelfand_discriminants_1994}, to the $\mathscr{A}$\textit{-discriminant} $D_{\mathscr{A}}\subset\mathbb{C}_{\mathscr{A}}=\mathbb{C}_{A_{1}}\oplus\cdots\oplus\mathbb{C}_{A_{k}}$, that is, the closure of the set of all tuples of polynomials $(f_{1},\ldots,f_{k})$ such that the system of equations $f_{1}=\cdots=f_{k}=0$ has a \textit{degenerate root} (a point $x\in T$ at which $f_{1}(x)=\cdots=f_{k}(x)=0$ and $df_{1}\wedge\cdots\wedge df_{k}(x)=0$). The motivation varies from algebraic geometry, algebraic statistics \cite{huh_likelihood_2014,amendola_maximum_2019}, and PDEs \cite{gelfand_hypergeometric_1989,gelfand_generalized_1990} to mathematical physics \cite{vanhove_feynman_2019,mizera_landau_2022,matsubara-heo_four_2023} and symbolic computation \cite{sturmfels_solving_2002}.
	
	A question of particular interest is the classification of the tuples $\mathscr{A}$ for which the discriminant is not a hypersurface. For $k=1$, this is the classical problem of dual defective varieties (varieties whose projective dual is not a hypersurface), restricted to toric varieties; it was solved in a series of papers \cite{di_rocco_projective_2006,curran_restriction_2007,dickenstein_tropical_2007,esterov_newton_2010,matsui_geometric_2011,furukawa_combinatorial_2021,cattani_non-splitting_2022}. At the other extreme, for $n=k$, the discriminant was shown to have a hypersurface component whenever the support sets $A_{1},\ldots,A_{n}$ cannot be moved by shifts and an automorphism of the lattice into an affine plane or into the tuple of standard simplices. This was proved by Esterov \cite{esterov_galois_2019} and by Borger and Nill \cite{borger_defectivity_2020}, motivated by applications to Galois theory and to the geometry of lattice polytopes, respectively.
	
	Two further types of discriminants can be defined. For a tuple $\mathscr{A}$, the \textit{mixed discriminant} is the closure of the set of all polynomial systems in $\mathbb{C}_{\mathscr{A}}$ that have a \textit{nondegenerate multiple root}, that is, a degenerate root $x\in T$ such that no proper subtuple of $df_{1}(x),\dots,df_{k}(x)$ is linearly dependent. Mixed discriminants were introduced by Cattani, Cueto, Dickenstein, Di Rocco, and Sturmfels \cite{cattani_mixed_2013} and further studied in \cite{dokken_plane_2014,esterov_galois_2019,dickenstein_iterated_2023}.
	
	For a subset $I\subseteq\{1,\ldots,k\}$, the \textit{Cayley trick} for the subtuple $\mathscr{B}=(A_{i},\;i\in I)$ is the map sending a polynomial system $f\in\mathbb{C}_{\mathscr{A}}$ to the polynomial $\sum_{i\in I}\lambda_{i}f_{i}(x)$ in the variables $x$ and $\lambda$. The support of this polynomial is the \textit{Cayley set} $\cay(\mathscr{B})=\bigcup_{i\in I}A_{i}\times\{i\}\subset M\times\mathbb{Z}^{|I|}$. The \textit{Cayley discriminant} of the subtuple $\mathscr{B}$ is the preimage under the Cayley trick of the discriminant of the Cayley set $\cay(\mathscr{B})$ \cite{esterov_newton_2010,esterov_galois_2019}. Cayley discriminants were introduced as intermediate objects that reduce the study of $\mathscr{A}$-discriminants of polynomial systems to that of $A$-discriminants of single polynomials, which are much simpler and better understood.
	
	In the case $n=k$, if the Cayley discriminant is a hypersurface, then the mixed discriminant is the same hypersurface \cite{cattani_mixed_2013}. For the tuples called irreducible, Esterov showed that the three types of discriminants share a hypersurface component and conjectured that the $\mathscr{A}$-discriminant has no other components \cite{esterov_galois_2019}. We prove this conjecture; consequently, for an irreducible tuple the three types of discriminants coincide and form a hypersurface.
	
	The following result, which is of independent interest, is crucial for the proof of Esterov's irreducibility conjecture. Let $V$ be an $n$-dimensional vector space over an algebraically closed field, equipped with the Zariski topology. A tuple of vector subspaces $(L_{1},\dots,L_{n})$ of $V$ is \textit{irreducible} if, for $0<k<n$, no $k$ of them lie in a common $k$-dimensional subspace. The product $V^{\times n}$ is isomorphic to the space of $n\times n$ matrices, which contains the determinant hypersurface, the set of matrices with zero determinant.
	
	\begin{thm}[Theorem \ref{Theorem. Intersection of the determinant with a subspace}]
		A tuple $(L_{1},\dots,L_{n})$ of vector subspaces in an $n$-dimensional vector space over an algebraically closed field is irreducible if and only if the intersection of the determinant hypersurface with the subspace $L_{1}\times\dots\times L_{n}$ is irreducible in the space of $n\times n$ square matrices.
	\end{thm}
	
	The proof is based on the polymatroid partition of the dual vector
	space \cite{pokidkin_combinatorics_2025}. 
	
	When $1<k<n$, the study of discriminants is substantially more complicated than in the classical case $k=1$ and in the case $n\leq k$ considered here; see \cite{esterov_newton_2010,esterov_discriminant_2013,dickenstein_iterated_2023} for partial results.
	
	We study the case $n\leq k$ comprehensively: for arbitrary support sets $\mathscr{A}=(A_{1},\ldots,A_{k})\subset M$, we give the complete list of the irreducible components of the $\mathscr{A}$-discriminant, the Cayley discriminant, and the mixed discriminant in $\mathbb{C}_{\mathscr{A}}$, together with their codimensions and degrees, except for one case of Cayley discriminants, which reduces to underdetermined systems. It turns out that the realizable polymatroid associated with the tuple $\mathscr{A}$ encodes each of the three types of discriminants as well as the resultant. This link between the theory of polymatroids and algebraic geometry differs from the one established in \cite{pagaria_hodge_2023,crowley_bergman_2026}. 
	
	The answer is stated in terms of the following basic quantity: the \textit{defect} of a subtuple $\mathscr{B}=(A_{i},\;i\in I)$ is the number $\delta(\mathscr{B})=\dim($affine span of the Minkowski sum $\sum_{i\in I}A_{i})-|I|$. A tuple is \textit{dependent} if it contains a subtuple of negative defect. A tuple is a \textit{circuit} if it is a minimal dependent tuple, and a \textit{cycle} is a union of circuits.
	
	Throughout the paper, we assume that every finite set $A\in\mathscr{A}$ contains the origin and that the sets of the tuple $\mathscr{A}$ generate the lattice $M$. As shown below, these assumptions do not affect the generality.
	
	We first describe the discriminants of dependent tuples: this simplest case reduces to the sparse resultant, introduced in \cite{sturmfels_newton_1994}, for the maximal cycle in Theorem \ref{Theorem. A/M is linearly independent}. 
	\begin{thm}\label{Theorem. Introduction. Discriminants for linearly dependent tuples}
		Let $\mathscr{A}$ be a dependent tuple, and let $\mathscr{M}$ be its maximal cycle. Then
		\begin{enumerate}[label=\textup{(}\textit{\alph*}\textup{)}]
			\item the $\mathscr{A}$-discriminant is the sparse resultant $R_{\mathscr{M}}$, which has codimension $-\delta(\mathscr{M})$ \textup{(Corollary~\ref{Corollary. Discriminants for linearly dependent tuples})}; 
			\item if the tuple $\mathscr{A}$ is a circuit, the mixed discriminant is the sparse resultant $R_{\mathscr{A}}$, which is a hypersurface; otherwise the mixed discriminant is empty \textup{(Theorem \ref{Theorem. The mixed discriminant for a linearly dependent tuple})}.
		\end{enumerate}
	\end{thm}
	
	The case of independent tuples is the essence of the matter, and we begin with irreducible tuples. A \textit{proper} subtuple of a tuple is a subtuple that is neither the whole tuple nor the empty tuple. A tuple is \textit{irreducible} if all its proper subtuples have positive defect. Independent tuples of zero defect are called \textit{BK-tuples}. A tuple is \textit{linear} if it can be embedded into the tuple of standard simplices by an automorphism of the lattice.
	
	\begin{thm}[Esterov {\cite[Conjecture 3.15]{esterov_galois_2019}}; Theorem \ref{Theorem. Codimension of discriminants}]\label{Theorem. Esterov conjecture}
		For a nonlinear irreducible BK-tuple
		$\mathscr{A}$, the $\mathscr{A}$-discriminant is a hypersurface
		in the space of polynomial systems $\mathbb{C}_{\mathscr{A}}$.
	\end{thm}
	
	For an irreducible BK-tuple $\mathscr{A}$, the monodromy group is the symmetric group on $\MV(\mathscr{A})$ elements, where $\MV(\mathscr{A})$ is the mixed volume \cite{esterov_galois_2019}, and the general polynomial system is solvable by radicals if the mixed volume $\MV(\mathscr{A})$ does not exceed four \cite{esterov_multivariate_2016}. For reducible BK-tuples these questions remain open, and the present work provides a foundation for their solution.

	The general description of the discriminants of square polynomial systems is based on Theorem \ref{Theorem. Esterov conjecture} and requires the following notions.
	\begin{defn}
		A BK-tuple is \textit{simple} if it is not a union of its proper BK-subtuples. A simple BK-subtuple is \textit{maximal} if it is not contained in any other simple BK-subtuple.
		
		A simple BK-tuple is \textit{prelinear} if its image under a projection of lattices that sends every set of its unique maximal proper BK-subtuple to zero is a linear tuple.
	\end{defn}
	
	\begin{thm}[Theorem \ref{Theorem. Discriminant for a BK-tuple}
		and Proposition \ref{Proposition. Discriminant's strata are Cayley discriminants}] \label{Theorem. Introduction. Discriminants for linearly independent tuples}
		For a BK-tuple $\mathscr{A}$, the $\mathscr{A}$-discriminant has the following distinct components:
		\begin{enumerate}[label=\textup{(}\textit{\alph*}\textup{)}]
			\item the Cayley discriminants of the nonprelinear simple BK-subtuples, which have codimension one;
			\item the Cayley discriminants of the prelinear simple BK-subtuples that are not contained in nonprelinear simple BK-subtuples, which have codimension two.
		\end{enumerate}
	\end{thm}
	
	\begin{rem}
		(1) If a prelinear BK-subtuple is contained in a nonprelinear BK-subtuple, then the Cayley discriminant of the former lies in the Cayley discriminant of the latter.
		
		(2) The proof is based on a specific partition of the tuple $\mathscr{A}$ given in Corollary \ref{Corollary. Unique decomposition of a BK-tuple}. This partition originates from a partition of the realizable polymatroid on $\mathscr{A}$ \cite{pokidkin_combinatorics_2025}.
	\end{rem}
	
	\begin{thm}[Theorem \ref{Theorem. The mixed discriminant for a BK-tuple}; \cite{cattani_mixed_2013}]
		For a BK-tuple, the mixed discriminant is empty, a hypersurface, or a variety of codimension two according to whether the tuple is nonsimple, nonprelinear simple, or prelinear simple. For a simple BK-tuple, the mixed discriminant coincides with the Cayley discriminant.
	\end{thm}
	
	Although Cayley discriminants are well-known $A$-discriminants, the present paper provides an alternative description of the Cayley discriminants of BK-tuples in terms of maximal simple BK-subtuples.
	
	\begin{thm}[Theorem \ref{Theorem. Cayley Discriminants}]
		For a BK-tuple, the Cayley discriminant has codimension $\mathfrak{n}+2\mathfrak{l}$, where $\mathfrak{n}$ and $\mathfrak{l}$ are the numbers of maximal simple BK-subtuples that are nonprelinear and prelinear, respectively.
	\end{thm}
	
	The known results \cite{pedersen_product_1993,gelfand_discriminants_1994,dickenstein_tropical_2007,esterov_determinantal_2007,matsui_geometric_2011}
	allow us to compute the degrees of the discriminants. The results are collected in Table \ref{Table. Degrees of discriminants}. 
	
	The defining equations of discriminants can be found using computer algebra systems (e.g., \href{https://symbolics.juliasymbolics.org}{Symbolics.jl},
	\href{https://macaulay2.com}{Macaulay2}, \href{https://www.oscar-system.org}{OSCAR},
	or \href{https://www.sagemath.org}{SageMath}) and then used to enumerate the components of the discriminants and to determine their geometric characteristics. However, the complexity of this computation grows rapidly with the dimension and with the size of the support sets. Our results express the geometry of the components of the discriminants combinatorially, in terms of the support sets, thereby circumventing the costly symbolic computation of discriminants. 
	
	The paper is organized as follows. Section \ref{Section. Irreducible intersections of subspaces with determinant} characterizes the row-generated subspaces of the space of square matrices whose intersection with the determinant hypersurface is irreducible. Section \ref{Section. Esterov irreducibility theorem} is devoted to the Esterov conjecture. Sections \ref{Section. Combinatorial Review} and \ref{Section. Discriminants-of-polynomial} recall combinatorial results and general facts about discriminants. In Section \ref{Section. BK-multiplicaiton}, we construct a special multiplication of varieties, which is needed to describe the discriminants of BK-tuples. Section \ref{Section. Discriminants for BK-tuples} describes $\mathscr{A}$-discriminants, Section \ref{Section. Cayley discriminants} Cayley discriminants, and Section \ref{Section. Mixed discriminants} mixed discriminants. For each type of discriminant, we enumerate the components and compute their codimensions, first for BK-tuples and then for dependent tuples. All computations of degrees are collected in Section \ref{Section. Degrees of components} and Table \ref{Table. Degrees of discriminants}. The descriptions of the discriminants are summarized in Table \ref{Table. Components of discriminants}.
	
	\begin{table}[htbp]
		\caption{Components of the three discriminants for dependent and BK-tuples}
		\label{Table. Components of discriminants}
		\small
		\setlength{\tabcolsep}{4pt}
		\renewcommand{\arraystretch}{1.25}
		\begin{tabular}{@{}>{\raggedright\arraybackslash}p{2.55cm}
				>{\raggedright\arraybackslash}p{3.95cm}
				>{\raggedright\arraybackslash}p{3.95cm}
				>{\raggedright\arraybackslash}p{3.95cm}@{}}
			\toprule
			Tuple $\mathscr{A}$
			& $\mathscr{A}$-discriminant
			& Cayley discriminant
			& Mixed discriminant \\
			\midrule
			% ---------- dependent ----------
			Dependent
			& variety of
			& variety, 
			& hypersurface\\
			
			($\mathscr{M}$ is the
			& codim $-\delta({\mathscr{M}})$,
			& if $\mathscr{A}/\mathscr{M}$ is BK, codim 
			& if $\mathscr{A}=\mathscr{M}$ is a circuit,\\
			
			max cycle)
			&
			& $-\delta({\mathscr{M}})+(\mathfrak{n}+2\mathfrak{l})(\mathscr{A}/\mathscr{M})$,
			& otherwise $\varnothing$,\\

			& Corollary \ref{Corollary. Discriminants for linearly dependent tuples}
			& Theorem \ref{Theorem. Cayley disccriminant for a BK-union}
			& Theorem \ref{Theorem. The mixed discriminant for a linearly dependent tuple}\\
			
			\midrule
			% ---------- BK ----------
			BK-tuple
			& at most $|\mathtt{P}_\mathscr{A}|$ components
			& variety of
			& variety of codim $1$ or $2$\\
			
			($\mathtt{P}_\mathscr{A}$ is the
			& of codim 1 and 2,
			& codim $(\mathfrak{n}+2\mathfrak{l})(\mathscr{A})$,
			& if $|\max\mathtt{P}_\mathscr{A}|=1$,\\
			
			Birkhoff poset)
			& 
			& determined by $\max\mathtt{P}_\mathscr{A},$
			& otherwise $\varnothing$,\\

			& Theorem \ref{Theorem. Discriminant for a BK-tuple}
			& Theorem \ref{Theorem. Cayley Discriminants}
			& Theorem \ref{Theorem. The mixed discriminant for a BK-tuple}
		\end{tabular}
	\end{table}
	
	\section{Irreducible intersections of subspaces with the determinant}\label{Section. Irreducible intersections of subspaces with determinant}
	
	For a quasi-affine algebraic set $X$, we denote its Zariski closure
	by $\overline{X}$. By a \textit{variety} we mean an irreducible quasi-affine algebraic set over an algebraically closed field. A \textit{hypersurface} is a variety of codimension one.
	
	Consider an $n$-dimensional vector space $V$ over an algebraically closed field $\mathbb{K}$ of arbitrary characteristic, equipped with the Zariski topology. Let $E$ be the algebraic set defined by the equation $\varphi(l)=0$ in the space $\mathrm{End}(V^{\vee})\times V^{\vee}\ni(\varphi,l)$, where $V^{\vee}$ is the dual space and $\mathrm{End}(V^{\vee})$ is the vector space of endomorphisms of $V^{\vee}$. In coordinates, the equation $\varphi(l)=0$ defines an intersection of $n$ quadrics in $\mathrm{End}(V^{\vee})\times V^{\vee}$. Every point $(\varphi,l)$ of $E$ with a nonzero covector $l$ corresponds to a degenerate linear operator $\varphi$ together with an eigenvector $l$ with eigenvalue zero. 
	
	The quadric equations can be written explicitly. Choose a basis $e_{1},\dots,e_{n}$ of the vector space $V$, the dual basis $e^{1},\dots,e^{n}$ of the dual space $V^{\vee}$, and the basis $e^{i}\otimes e_{j}$ of $\mathrm{End}(V^{\vee})$ provided by the isomorphism $\mathrm{End}(V^{\vee})\cong V^{\vee}\otimes V$. Using the Einstein summation convention, we write $l=l_{i}e^{i}$ and $\varphi=\varphi_{i}^{j}e^{i}\otimes e_{j}$ with $l_{i},\varphi_{i}^{j}\in\mathbb{K}$. Then the equation $\varphi(l)=0$ is equivalent to $\varphi_{i}^{j}l_{j}e^{i}=0$. Since the vectors $e^{i}$ form a basis, all coefficients $\varphi_{i}^{j}l_{j}$ must vanish, and we obtain an intersection of the $n$ quadrics $\varphi_{i}^{j}l_{j}=0$ in $\mathrm{End}(V^{\vee})\times V^{\vee}$.
	
	Note that the equations $\varphi_{i}^{j}l_{j}=0$ are equivalent to $l(x_{i})=0$, where $x_{i}=\varphi_{i}^{j}e_{j}\in V$. Hence there is an isomorphism $\mathrm{End}(V^{\vee})\cong V^{\times n}$ such that $\varphi(l)=l(x_{i})e^{i}$ for a tuple of vectors $(x_{1},\dots,x_{n})\in V\times\dots\times V$. The equations $l(x_{i})=0$ then define the intersection $E$ of $n$ quadrics in the space $V^{\times n}\times V^{\vee}$. Note that no choice of bases in $V$ and $V^{\vee}$ is needed to write the system $l(x_{i})=0$, $i\in[n]=\{1,\dots,n\}$, in $V^{\times n}\times V^{\vee}\ni(x_{1},\dots,x_{n},l)$.
	
	There are two natural projections: 
	\[
	\begin{array}[t]{cc}
		\xymatrix{ & E\ar@{->>}[ld]_{p}\ar@{->>}[rd]^{q}\\
			V^{\times n} & & V^{\vee}
		}
		& \begin{array}[t]{l}
			p(x_{1},\dots,x_{n},l)=(x_{1},\dots,x_{n}),\\
			q(x_{1},\dots,x_{n},l)=l.
	\end{array}\end{array}
	\]
	For a nonzero covector $l$, the fiber $q^{-1}(l)$ is a vector subspace of $V^{\times n}$ of dimension $n^{2}-n$. For a matrix $\varphi$, the fiber $p^{-1}(\varphi)$ is a subspace of $V^{\vee}$ whose dimension equals the geometric multiplicity of the eigenvalue zero. A similar construction appears in \cite[Example 12.1]{harris_algebraic_1992}.
	
	Recall that a \textit{polymatroid} is a pair $([n],\rk)$ consisting of a finite set $[n]$ and a rank function $\rk\colon 2^{[n]}\rightarrow\mathbb{Z}_{\geq0}$ that is submodular, monotone, and normalized ($\rk(\varnothing)=0$). A tuple of subspaces $(L_{1},\dots,L_{n})$ of the vector space $V$ together with the rank function $\rk_{P}(I)=\dim L_{I}$, $L_{I}=\sum_{i\in I}L_{i}$, forms a polymatroid $P$. The same polymatroid $P$ is obtained from the tuple of orthogonal complements $(L_{1}^{\perp},\dots,L_{n}^{\perp})$ in the dual space $V^{\vee}$ with the rank function $\rk_{P}^{\perp}(I)=\codim L_{I}^{\perp}$,
	$L_{I}^{\perp}=\bigcap_{i\in I}L_{i}^{\perp}$. 
	
	The \textit{defect} of a set $I\subseteq[n]$ is the number $\delta(I)=\rk_{P}(I)-|I|$.
	
	A \textit{flat} of the polymatroid $P$ is a subset $F\subseteq[n]$
	that is maximal among the sets of the same rank. The set of all flats, ordered by inclusion, forms a lattice $\mathcal{L}$. By \cite{pokidkin_combinatorics_2025}, the dual space $V^{\vee}$ admits a polymatroid partition into constructible sets $B_{F}$ indexed by the flats $F\in\mathcal{L}$; namely, $B_{F}=L_{F}^{\perp}\setminus\bigcup_{F'\in\mathcal{L}\setminus(F)}L_{F'}^{\perp}$, where $(F)=\{F'\in \mathcal{L}\mid F'\subseteq F\}$ is the principal order ideal of $F$ in the lattice of flats $\mathcal{L}$. The sets $B_{F}$ are nonempty and dense in the subspaces $L_{F}^{\perp}$, because the field $\mathbb{K}$ is algebraically closed and hence infinite. 
	
	Denote by $L$ the product $L_{1}\times\dots\times L_{n}\subseteq V^{\times n}$ and, for a covector $l$, by $E_{l}|_{L}$ the restriction $q^{-1}(l)|_{L}=p(q^{-1}(l))\cap L$ of the fiber $q^{-1}(l)$ to the subspace $L$.
	
	\begin{lem}\label{Lemma. Dimension of a fiber}
		For every $l\in B_{F}$, the restriction has dimension $\dim E_{l}|_{L}=\dim L-n+|F|$.
	\end{lem}
	\begin{proof}
		The constructible set corresponding to a flat $F$ is $B_{F}=L_{F}^{\perp}\setminus\bigcup_{F'\in\mathcal{L}\setminus(F)}L_{F'}^{\perp}$. This can be rewritten as
		\begin{align*}
			B_{F} & =\{l\in V^{\vee}\mid l\in L_{F}^{\perp}\text{ and }l\notin L_{j}^{\perp},\;j\in[n]\setminus F\}\\
			& =\{l\in V^{\vee}\mid l^{\perp}\supseteq L_{F}\text{ and }l^{\perp}\not\supseteq L_{j},\;j\in[n]\setminus F\}.
		\end{align*} In other words, $B_{F}$ consists of the covectors $l$ whose kernel $l^{\perp}$ contains the subspace $L_{i}$ for every index $i$ in the flat $F$ and contains none of the subspaces $L_{j}$, $j\in[n]\setminus F$. Therefore, the dimension of the restriction $E_{l}|_{L}=(l^{\perp}\times\dots\times l^{\perp})\cap L=(l^{\perp}\cap L_{1})\times\dots\times(l^{\perp}\cap L_{n})$
		to the subspace $L$ equals 
		\begin{align*}
			\dim E_{l}|_{L} & =\sum_{i=1}^{n}\dim l^{\perp}\cap L_{i}=\sum_{i=1}^{n}(\dim l^{\perp}+\dim L_{i}-\dim(l^{\perp}+L_{i}))\\
			& =(n^{2}-n)+\dim L-(n^{2}-n+n-|F|)=\dim L-n+|F|.
		\end{align*}
	\end{proof}
	Denote by $Q_{F}$ the intersection $q^{-1}(B_{F})\cap(L\times V^{\vee})$. The polymatroid partition $V^{\vee}=\bigsqcup_{F\in\mathcal{L}}B_{F}$ and Lemma \ref{Lemma. Dimension of a fiber} yield the partition
	\[
	E\cap(L\times V^{\vee})=\bigsqcup_{F\in\mathcal{L}}Q_{F}.
	\]
	
	\begin{prop}\label{Proposition. Vector bundles other B_F}
		For a proper flat $F$, the triple $(Q_{F},q,B_{F})$ is a vector bundle of rank $r=\dim L-n+|F|$, and $Q_{F}$ is a variety of dimension $\dim L-\delta(F)$.
	\end{prop}
	\begin{proof}
		By Lemma \ref{Lemma. Dimension of a fiber}, every fiber is a vector space of dimension $r$. Therefore, the dimension of the preimage $Q_{F}$ equals 
		\[
		\dim Q_{F}=\dim B_{F}+\dim \text{fiber}=(n-\dim L_{F})+(\dim L-n+|F|)=\dim L-\delta(F).
		\]
		
		Let us show that the base $B_{F}$ can be covered by open charts $Y_{F}$ such that the preimages $q^{-1}(Y_{F})$ are isomorphic to the trivial vector bundles $Y_{F}\times\mathbb{K}^{r}$. For each $i\in[n]\setminus F$, choose an arbitrary vector $v_{i}\in L_{i}\setminus L_{F}$, and define the set $Y_{F}=B_{F}\setminus\bigcup_{i\in[n]\setminus F}v_{i}^{\perp}$, which is open in the vector subspace $L_{F}^{\perp}$. For each $i\in[n]\setminus F$, choose an arbitrary decomposition $L_{i}=U_{i}\oplus\langle v_{i}\rangle$ of the subspace $L_{i}$ into a direct sum. Note that the preimage $q^{-1}(Y_{F})$ is an open subset of $Q_{F}$ and equals the intersection $E\cap(L\times Y_{F})$. By the construction of $B_{F}$, the preimage $Q_{F}$ is the intersection of the set $L\times B_{F}$ with the $n-|F|$ quadrics $l(x_{i})=0$, $x_{i}\in L_{i}$, $i\in[n]\setminus F$, $l\in B_{F}$. Restricting these equations to the set $L\times Y_{F}$ and using the decomposition $L_{i}=U_{i}\oplus\langle v_{i}\rangle$, $x_{i}=u_{i}+\lambda_{i}v_{i}$, $u_{i}\in U_{i}$, $\lambda_{i}\in\mathbb{K}$, we obtain
		\[
		l(u_{i})+\lambda_{i}l(v_{i})=0,\qquad l\in Y_{F}.
		\]
		The value $l(v_{i})$ is nonzero for every $l\in Y_{F}$ and every $i$, so the variables $\lambda_{i}$ can be eliminated using these equations. Hence the coordinate algebra of the preimage $q^{-1}(Y_{F})$ is 
		\[
		\mathbb{K}[q^{-1}(Y_{F})]\cong\mathbb{K}\Bigl[\Bigl(\bigoplus_{i\in F}L_{i}\oplus\bigoplus_{j\in[n]\setminus F}U_{j}\Bigr)\times Y_{F}\Bigr],
		\]
		which is a polynomial algebra over the ring $\mathbb{K}[Y_{F}]$. Therefore, the preimage $q^{-1}(Y_{F})$ is a variety and a trivial vector bundle over the open chart $Y_{F}$ with the expected fiber. The vectors $v_{i}$ and the subspaces $U_{i}$ were chosen arbitrarily. Let us show that the base $B_{F}$ is covered by the affine charts $Y_{F}=Y_{F}(v)$, $v=(v_{i})_{i\in[n]\setminus F}\in C_{F}=\prod_{i\in[n]\setminus F}L_{i}\setminus L_{F}$,
		\begin{align*}
			\bigcup_{v\in C_{F}}Y_{F}(v) & =\bigcup_{v\in C_{F}}(B_{F}\setminus\bigcup_{i\in[n]\setminus F}v_{i}^{\perp})=\bigcup_{v\in C_{F}}\bigcap_{i\in[n]\setminus F}B_{F}\setminus v_{i}^{\perp}=\bigcap_{i\in[n]\setminus F}(\bigcup_{v_{i}\in L_{i}\setminus L_{F}}B_{F}\setminus v_{i}^{\perp})\overset{(*)}{=}B_{F},\\
			(*) & \bigcup_{v_{i}\in L_{i}\setminus L_{F}}B_{F}\setminus v_{i}^{\perp}=B_{F}\setminus\bigcap_{v_{i}\in L_{i}\setminus L_{F}}v_{i}^{\perp}=B_{F}\setminus L_{i}^{\perp}=B_{F}.
		\end{align*}
		This computation shows that $Q_F$ is a variety, being a vector bundle over the variety $B_F$.
	\end{proof}
	
	\begin{prop}
		The intersection of the determinant hypersurface $\Det$ with the subspace $L$ in the space of square matrices is the union
		\[
		\Det\cap L=\bigcup_{F\in\mathcal{L}\setminus[n]}\overline{p(Q_{F})}.
		\]
	\end{prop}
	\begin{proof}
		The intersection $\Det\cap L$ is defined by the single polynomial equation $\det\varphi|_{L}=0$. Hence it is closed, and all its components have codimension one in the subspace $L$.
		
		By the polymatroid partition of the dual vector space \cite{pokidkin_combinatorics_2025} and Proposition \ref{Proposition. Vector bundles other B_F}, the intersection $E\cap(L\times V^{\vee})=\bigcup_{F\in\mathcal{L}}Q_{F}$ is the union of the strata $Q_{F}$ indexed by the flats $F$ of the polymatroid on $(L_{1},\dots,L_{n})$ with lattice of flats $\mathcal{L}$. By definition, the zero vector is not an eigenvector, so every point $(\varphi,l)$ of $E\setminus q^{-1}(0)$ is a pair consisting of a degenerate matrix $\varphi$ and a nonzero eigenvector $l$ with eigenvalue zero. Hence the determinant hypersurface $\Det$ is the closure of the projection $p(E\setminus q^{-1}(0))$. Observe that $B_{[n]}=L_{[n]}^{\perp}=(L_{[n]})^{\perp}=V^{\perp}=\{0\}$,
		$Q_{[n]}=q^{-1}(B_{[n]})\cap(L\times V^{\vee})=L\times\{0\}$, and the intersection $\Det\cap L$ is the union
		\[
		\Det\cap L=\overline{p((E\setminus q^{-1}(0))\cap(L\times V^{\vee}))}=\overline{p((E\cap(L\times V^{\vee}))\setminus Q_{[n]})}=\bigcup_{F\in\mathcal{L}\setminus[n]}\overline{p(Q_{F})},
		\]
		where $\mathcal{L}\setminus[n]$ denotes the lattice of flats with the maximal element $[n]$ removed. 
	\end{proof}
	Recall that a tuple of subspaces $(L_{1},\dots,L_{n})$ is \textit{independent} if all its subtuples have nonnegative defect. By Rado's theorem \cite{rado_theorem_1942}, this is equivalent to the existence of $n$ linearly independent vectors $v_{i}\in L_{i}$.
	\begin{cor}
		For a dependent tuple $(L_{1},\dots,L_{n})$ of vector subspaces, the subspace $L$ lies in the determinant hypersurface $\Det$.
	\end{cor}
	
	\begin{thm}\label{Theorem. Intersection of the determinant with a subspace}
		A BK-tuple $(L_{1},\dots,L_{n})$ of vector subspaces is irreducible if and only if the intersection of the determinant hypersurface $\Det$ with the subspace $L$ is a variety.
	\end{thm}
	\begin{proof}
		$(\Rightarrow)$ The equation $\det\varphi|_{L}=0$ is nontrivial, because irreducible tuples of subspaces are independent. Every matrix in the intersection $\Det\cap L$ has at least one eigenvector with eigenvalue zero, and the dimension of the projection is bounded by $\dim p(Q_{F})\leq \dim Q_{F}-1=\dim L-\delta(F)-1$.
		Since the tuple $(L_{1},\dots,L_{n})$ is irreducible, the proper flats have positive defect, and $\dim p(Q_{F})<\dim L-1$. The inequality $\dim p(Q_{\varnothing})\leq \dim L-1$ is not strict only for the empty BK-subtuple $(\varnothing)$, so no other stratum can have codimension one in $L$. Since the intersection $\Det\cap L$ is nonempty and has codimension one, we have $\codim_{L}p(Q_{\varnothing})=1$. Since the stratum $Q_{\varnothing}$ is a variety, so is its projection $p(Q_{\varnothing})$. The closure of $p(Q_{\varnothing})$ contains the projections of all other strata $Q_{F}$ and equals the intersection $\Det\cap L$.
		
		$(\Leftarrow)$ If the BK-tuple is reducible, then, after a change of basis, a proper BK-subtuple corresponds to a proper block in the space of square matrices, the subspace $L$ acquires a block-triangular shape, and the determinant $\det |_L$ factors.
	\end{proof}
	\begin{rem}
		For a BK-tuple $(L_{1},\dots,L_{n})$ of vector subspaces, one can show that the number of components of the intersection $\Det\cap L$ equals the number of elements of the Birkhoff poset of BK-subtuples \cite{pokidkin_combinatorics_2025}. This completes the description of the intersection of the determinant hypersurface with a row-generated subspace.
	\end{rem}

	\section{Esterov irreducibility theorem}
	\label{Section. Esterov irreducibility theorem}
	
	From now on, we work over the field of complex numbers. For a BK-tuple $\mathscr{A}$, we look for solutions of polynomial systems $\Phi_{\mathscr{A}}$ in the algebraic torus $T(\mathscr{A})$. Denote by $\mathfrak{c}(\mathscr{A})$ the number of sets in the tuple $\mathscr{A}$ and by $\langle A\rangle$ the sublattice generated by a set $A$.
	\begin{thm}\label{Theorem. Esterov Conjecture}
		For an irreducible BK-tuple $\mathscr{A}$, the discriminant $D_{\mathscr{A}}$ is a variety in $\mathbb{C}_{\mathscr{A}}$.
	\end{thm}
	\begin{proof}
		The tuple $\mathscr{A}$ defines a vector subspace of matrices $\Mat_{\mathscr{A}}=\bigoplus_{A\in\mathscr{A}}(\langle A\rangle\otimes\mathbb{C})$ in the space of matrices $\Mat_{\dim\langle\mathscr{A}\rangle,\mathfrak{c}(\mathscr{A})}$. Since the BK-tuple $\mathscr{A}$ is irreducible, the intersection $Y$ of the determinant hypersurface $\Det$ with the subspace $\Mat_{\mathscr{A}}$ is a variety by Theorem \ref{Theorem. Intersection of the determinant with a subspace}.
		
		With a finite set $A$ and with the tuple $\mathscr{A}$, we associate the integer matrices $\tau_{A}=(a,\;a\in A)$ and $\tau_{\mathscr{A}}=\bigoplus_{A\in\mathscr{A}}\tau_{A}$, whose ranks are $\rk\tau_{A}=\dim \langle A\rangle$ and $\rk\tau_{\mathscr{A}}=\dim \Mat_{\mathscr{A}}$. The matrix $\tau_{\mathscr{A}}$ defines a linear map from $\mathbb{C}_{\mathscr{A}}$ to $\Mat_{\mathscr{A}}$ by the formula $\tau_{\mathscr{A}}(\lambda_{\mathscr{A}})=(\tau_{A}(\lambda_{A}),\;A\in\mathscr{A})$, where $\tau_{A}(\lambda_{A})=\sum_{a\in A}\lambda_{a}\,a$.
		Consequently, the preimage $K=\tau_{\mathscr{A}}^{-1}(Y)$ is a variety in $\mathbb{C}_{\mathscr{A}}$, being a trivial vector bundle over $Y$ of rank $\dim \mathbb{C}_{\mathscr{A}}-\dim \Mat_{\mathscr{A}}$.
		
		Consider the following sequence of maps:
		\[
		\begin{array}{ccccccc}
			T(\mathscr{A})\times\mathbb{C}_{\mathscr{A}} & \overset{\sigma_{\mathscr{A}}}{\twoheadrightarrow} & \mathbb{C}_{\mathscr{A}} & \overset{\tau_{\mathscr{A}}}{\rightarrow} & \Mat_{\mathscr{A}} & \overset{\det}{\rightarrow} & \mathbb{C},\end{array}
		\]
		where $\sigma_{\mathscr{A}}(x,c)=(c_{a}x^{a})_{a\in A\in\mathscr{A}}$ is the torus action. The preimage $W=\sigma_{\mathscr{A}}^{-1}(K)$ is a variety isomorphic to $T(\mathscr{A})\times K$; indeed, the isomorphism is given by the regular maps $W\stackrel[\iota]{\sigma}{\rightleftarrows}T(\mathscr{A})\times K$ such that $\iota(x,c_{a})=(x,c_{a}x^{-a})$ and $\sigma(x,c_{a})=(x,c_{a}x^{a})$ for every $a\in A\in\mathscr{A}$. The composition of these maps gives the condition for a root $x$ of a polynomial system $c=(c_{a})_{a\in A\in\mathscr{A}}$ to be a degenerate root:
		$\det\circ\tau_{\mathscr{A}}\circ\sigma_{\mathscr{A}}(x,c)=\det\|\sum_{a\in A}c_{a}x^{a}a\|_{A\in\mathscr{A}}=0$.
		This polynomial is irreducible because $W$ is a variety and the determinant is multilinear in the rows.
		
		Consider the vector subspaces $\Pi_{A}$ of $\mathbb{C}_{\mathscr{A}}$ defined by the linear equations $\sum_{a\in A}c_{a}=0$, and their intersection $\Pi=\bigcap_{A\in\mathscr{A}}\Pi_{A}$. Denote by $c_{A}$ the coefficient corresponding to the origin of the finite set $A$. From each linear equation we can express the variable $c_{A}=-\sum_{a\in A\setminus\{0\}}c_{a}$. The determinant $\det\|\sum_{a\in A}c_{a}a\|_{A\in\mathscr{A}}$ does not depend on the coefficients $c_{A}$, $A\in\mathscr{A}$. Hence the intersection $X=K\cap\Pi$ of the variety $K$ with the subspace $\Pi$ is a variety defined by the same determinant. Indeed, the coordinate algebra $\mathbb{C}[X]$ is an integral domain:
		\[
		\mathbb{C}[X]=\frac{\mathbb{C}[c_{a},\,a\in A\in\mathscr{A}]}{\langle\sum_{a\in A}c_{a},\;A\in\mathscr{A};\;\det\|\sum_{a\in A}c_{a}a\|_{A\in\mathscr{A}}\rangle}=\frac{\mathbb{C}[c_{a},\,a\in A\setminus\{0\},\,A\in\mathscr{A}]}{\langle \det\|\sum_{a\in A}c_{a}a\|_{A\in\mathscr{A}}\rangle}.
		\]
		
		Thus the preimage $Z=\sigma_{\mathscr{A}}^{-1}(X)$ is a variety, isomorphic to the direct product $T(\mathscr{A})\times X$ via the same maps $\iota$ and $\sigma$. Each preimage $\sigma_{\mathscr{A}}^{-1}(\Pi_{A})$ is the zero locus of the equation $f_{A}(x)=0$ in $T(\mathscr{A})\times\mathbb{C}_{\mathscr{A}}$. Hence the variety $Z$ projects onto a dense subset of the discriminant $D_{\mathscr{A}}$ under the map $T(\mathscr{A})\times\mathbb{C}_{\mathscr{A}}\overset{\pi}{\rightarrow}\mathbb{C}_{\mathscr{A}}$,
		$\pi(x,c)=c$:
		\begin{align*}
			Z=\sigma_{\mathscr{A}}^{-1}(X) & =\sigma_{\mathscr{A}}^{-1}(K\cap\Pi)=\sigma_{\mathscr{A}}^{-1}(K)\cap\sigma_{\mathscr{A}}^{-1}(\Pi)=W\cap\bigcap_{A\in\mathscr{A}}\sigma_{\mathscr{A}}^{-1}(\Pi_{A})\\
			& =V(\det\|\sum_{a\in A}c_{a}x^{a}a\|_{A\in\mathscr{A}})\cap\bigcap_{A\in\mathscr{A}}V(f_{A}(x)).
		\end{align*}
		Therefore, the discriminant $D_{\mathscr{A}}$, being the closure of the image of a variety, is a variety.
	\end{proof}
	\begin{defn}
		A linear (respectively, nonlinear) irreducible BK-tuple is called a \textit{lir} (respectively, a \textit{nir}).
	\end{defn}
	
	\begin{rem}
		Linear BK-tuples are the only tuples with unit mixed volume \cite{cattani_mixed_2013,esterov_systems_2015} when $\langle\mathscr{A}\rangle=M$. The discriminants of linear tuples were studied in \cite{esterov_galois_2019,borger_defectivity_2020}.
	\end{rem}
	
	\begin{thm}\label{Theorem. Codimension of discriminants}
		In the space $\mathbb{C}_{\mathscr{A}}$, the discriminant $D_{\mathscr{A}}$ is a variety of codimension one if the tuple $\mathscr{A}$ is a nir, and of codimension two if the tuple $\mathscr{A}$ is a lir.
	\end{thm}
	\begin{proof}
		If the tuple $\mathscr{A}$ is a nir, then the discriminant $D_{\mathscr{A}}$ has a component of codimension one in $\mathbb{C}_{\mathscr{A}}$ by \cite{esterov_galois_2019}. By Theorem \ref{Theorem. Esterov Conjecture}, this component is unique, and the discriminant $D_{\mathscr{A}}$ is a variety.
		
		If the tuple $\mathscr{A}$ is a lir, then $\mathbb{C}_{\mathscr{A}}$ is a space of coefficients of systems of linear equations. In this case, the discriminant is the intersection of the subspace $\mathbb{C}_{\mathscr{A}}$ with the determinantal variety of matrices of rank at most $\mathfrak{c}(\mathscr{A})-1$ in the space of matrices of size $\mathfrak{c}(\mathscr{A})\times(\mathfrak{c}(\mathscr{A})+1)$. This intersection is irreducible by Theorem \ref{Theorem. Esterov Conjecture}, and it has codimension two, being an intersection of two determinantal hypersurfaces.
	\end{proof}
	
	\begin{cor}\label{Corollary. Generic set of solutions for a system of discriminant}
		The set of degenerate roots of a generic polynomial system of $D_{\mathscr{A}}$ has dimension zero if the tuple $\mathscr{A}$ is a nir, and dimension one if the tuple $\mathscr{A}$ is a lir. 
	\end{cor}
	\begin{proof}
		Consider the dominant map $Z\overset{\pi}{\rightarrow}D_{\mathscr{A}}$
		from the proof of Theorem \ref{Theorem. Esterov Conjecture}. For a system $\Phi\in D_{\mathscr{A}}$, the fiber $\pi^{-1}(\Phi)$ is isomorphic to the set of degenerate roots of $\Phi$. The generic fiber $\pi^{-1}(\Phi)$ has dimension $\dim Z-\dim D_{\mathscr{A}}$, which is determined by Theorem \ref{Theorem. Codimension of discriminants}.
	\end{proof}
	
	\begin{cor}[\cite{cattani_mixed_2013,esterov_galois_2019}]
		\label{Corollary. Irreducible BK => three types of discriminant are equal}
		For a nonlinear irreducible BK-tuple $\mathscr{A}$, the $\mathscr{A}$-discriminant, the Cayley discriminant, and the mixed discriminant coincide and form a hypersurface in $\mathbb{C}_{\mathscr{A}}$.
	\end{cor}
	\begin{proof}
		In this setting, the paper \cite{cattani_mixed_2013} established the equivalence between the mixed discriminant and the Cayley discriminant. Esterov proved that the three types of discriminants coincide up to codimension two \cite{esterov_galois_2019}. Theorem \ref{Theorem. Codimension of discriminants} shows that there are no other components other than a single hypersurface.
	\end{proof}
	
	\section{Combinatorics behind discriminants}\label{Section. Combinatorial Review}
	
	This section gives an overview of combinatorial results on configurations of sublattices, tuples of finite sets, and mixed volumes. For a tuple of sublattices $\mathsf{n}=(S_{1},\dots,S_{k})$ of a lattice $M$, the notion of irreducibility is motivated by questions of algebraic geometry. In the simplest case of a tuple of $n$ sublattices of $\mathbb{Z}^{n}$, a reducible tuple admits a partition into subtuples corresponding to irreducible tuples. This partition of the tuple $\mathsf{n}$ is determined by the realizable polymatroid on $\mathsf{n}$, and it is a combinatorial shadow of the decomposition of the $\mathscr{A}$-discriminant of a polynomial system into irreducible components.
	
	A tuple of sublattices can be regarded as a realizable polymatroid, and general polymatroids admit the same characterizations and terminology as tuples of sublattices. For simplicity, and since they are not needed here, we do not state results for general polymatroids; instead, we present the results of the general theory for tuples of sublattices. Note, however, that it was the study of discriminants that motivated these results in polymatroid theory and inspired the work \cite{pokidkin_combinatorics_2025}.
	\begin{defn}
		The \textit{saturation} of a sublattice $S$ is the maximal sublattice $\overline{S}$ containing $S$ and having the same dimension as $S$.
		
		The \textit{linear span} $\langle\mathsf{n}\rangle$ of a tuple of sublattices $\mathsf{n}$ is the smallest sublattice containing all sublattices of the tuple. The \textit{cardinality} $\mathfrak{c}(\mathsf{n})$ is the number of sublattices in the tuple $\mathsf{n}$, and the \textit{defect} of $\mathsf{n}$ is the difference $\delta(\mathsf{n})=\dim \langle\mathsf{n}\rangle-\mathfrak{c}(\mathsf{n})$.
		
		A tuple is \textit{essential} if every proper subtuple has strictly greater defect than the whole tuple. A tuple is \textit{independent} if all its subtuples have nonnegative defect. An independent tuple is \textit{irreducible} if all its proper subtuples have positive defect. A \textit{BK-tuple} is an independent tuple of zero defect.
		
		For a tuple $\mathsf{n}$ in a lattice $M$ and a subtuple $\mathsf{k}$ of $\mathsf{n}$, the \textit{quotient tuple} is the projection $\mathsf{n}/\mathsf{k}=\pi(\mathsf{n}\setminus\mathsf{k})$ of the complementary subtuple, where $\pi\colon M\rightarrow M/\overline{\langle\mathsf{k}\rangle}$.
	\end{defn}
	
	The notion of a BK-tuple arose from the Kouchnirenko--Bernshtein theorem \cite{bernshtein_number_1975}.
	\begin{prop}[\cite{pokidkin_combinatorics_2025}]
		For a reducible BK-tuple $\mathsf{n}$ and a BK-subtuple $\mathsf{k}$,
		\begin{enumerate}[label=\textup{(}\textit{\alph*}\textup{)}]
			\item the linear spans of any two distinct irreducible BK-subtuples intersect only at the origin;
			\item the quotient tuple $\mathsf{n}/\mathsf{k}$ is a BK-tuple;
			\item the BK-subtuples of the quotient $\mathsf{n}/\mathsf{k}$ are in bijection with the BK-subtuples of $\mathsf{n}$ containing $\mathsf{k}$;
			\item for a BK-subtuple $\mathsf{l}$ of $\mathsf{k}$, there is an isomorphism of quotient tuples $\mathsf{n}/\mathsf{k}\cong(\mathsf{n}/\mathsf{l})/(\mathsf{k}/\mathsf{l})$.
		\end{enumerate}
	\end{prop}
	
	\begin{defn}
		An \textit{order ideal} $\mathtt{I}$ of a poset $\mathtt{P}$ is a subset such that $\beta\in\mathtt{I}$ and $\alpha\leq\beta$ imply $\alpha\in\mathtt{I}$. For an element $\alpha$ of a poset, the \textit{principal order ideal} $(\alpha)$ is the order ideal consisting of all elements $\beta\leq\alpha$. The dual notions, obtained by reversing all inequalities, are \textit{order filter} and \textit{principal order filter} $[\alpha]$.
	\end{defn}
	
	Since intersections and unions of BK-subtuples of a BK-tuple $\mathsf{n}$ are again BK-subtuples \cite{pokidkin_combinatorics_2025}, the set of BK-subtuples ordered by inclusion forms a distributive lattice. By the fundamental theorem of finite distributive lattices \cite{stanley_enumerative_2011}, there exists a poset $\mathtt{P}$, called the \textit{Birkhoff poset}, whose lattice of order ideals is isomorphic to this distributive lattice.
	\begin{lem}\label{Lemma. The complement of a BK-subtuple corresponds to a filter}
		For a BK-tuple $\mathsf{n}$ and its subtuple $\mathsf{k}$, the complement $\mathsf{n}\setminus\mathsf{k}$ is a BK-tuple if and only if the subtuple $\mathsf{k}$ corresponds to an order filter in the Birkhoff poset of the tuple $\mathsf{n}$.
	\end{lem}
	\begin{proof}
		The tuple $\mathsf{n}\setminus\mathsf{k}$ is a BK-subtuple of $\mathsf{n}$ if and only if it corresponds to an order ideal, which is equivalent to the complementary subtuple $\mathsf{k}$ corresponding to an order filter of the Birkhoff poset of $\mathsf{n}$.
	\end{proof}
	\begin{defn}
		A \textit{filtration} of a tuple $\mathsf{n}$ is an increasing chain of subtuples $F_{0}\mathsf{n}\hookrightarrow F_{1}\mathsf{n}\hookrightarrow\dots\hookrightarrow F_{m}\mathsf{n}=\mathsf{n}$. A filtration is a \textit{BK-filtration} if all the quotients $F_{j}\mathsf{n}/F_{j-1}\mathsf{n}$ are BK-tuples, and a BK-filtration is \textit{maximal} if all these quotients are irreducible.
	\end{defn}
	
	\begin{prop}\label{Corollary. Filtration of BK-tuple with MV>0}
		For a BK-tuple $\mathsf{n}$, the number of subtuples in a maximal BK-filtration equals the number of elements of the Birkhoff poset. Moreover, for any two maximal BK-filtrations $F_{\bullet}\mathsf{n}$ and $F'_{\bullet}\mathsf{n}$, there exist linear isomorphisms between the successive quotients $F_{j}\mathsf{n}/F_{j-1}\mathsf{n}$ and $F_{j}'\mathsf{n}/F_{j-1}'\mathsf{n}$.
	\end{prop}
	
	The Birkhoff poset $\mathtt{P}$ defines a partition of the BK-tuple $\mathsf{n}$: every element $\alpha$ corresponds to a subtuple $\mathsf{k}_{\alpha}$ of $\mathsf{n}$, every principal order ideal $(\alpha)$ corresponds to the simple BK-subtuple $\mathsf{k}_{(\alpha)}=\bigsqcup_{\beta\leq\alpha}\mathsf{k}_{\beta}$, and every order ideal $\mathtt{I}$ of $\mathtt{P}$ corresponds to the BK-subtuple $\mathsf{k}_\mathtt{I}=\bigsqcup_{\alpha\in\mathtt{I}}\mathsf{k}_{\alpha}$. 
	
	\begin{thm}[\cite{pokidkin_combinatorics_2025}]\label{Theorem. Poset partition of a reducible BK-tuple}
		A BK-tuple $\mathsf{n}$ admits a unique partition $\mathsf{n}=\bigsqcup_{\alpha\in\mathtt{P}}\mathsf{k}_{\alpha}$ such that the tuple $\hat{\mathsf{k}}_{\alpha}=\mathsf{k}_{(\alpha)}/\mathsf{k}_{(\alpha)\setminus\alpha}$ is an irreducible BK-tuple for every element $\alpha$ of the poset $\mathtt{P}$.
	\end{thm}
	
	We now prove a few further statements needed for the description of Cayley discriminants. Recall that circuits always have defect $-1$ \cite{pokidkin_combinatorics_2025}.
	
	\begin{lem}\label{Lemma. BK-tuple, sublattice U and circuit}
		For a BK-tuple $\mathsf{n}$ of cardinality $n$ and a sublattice $S\subseteq\mathbb{Z}^n$, the tuple $(\mathsf{n}, S)$ corresponds to a circuit of the induced matroid if and only if $\dim S/\overline{\langle\mathsf{n}\setminus\mathsf{k}_\alpha\rangle}>0$ for every maximal element $\alpha$ of the Birkhoff poset $\mathtt{P}$ of the tuple $\mathsf{n}$.
	\end{lem}
	\begin{proof}
		$(\Rightarrow)$ For a maximal element $\alpha$, the tuple $\mathsf{n}\setminus\mathsf{k}_\alpha$ is a BK-subtuple with poset $\mathtt{P}\setminus\alpha$. Suppose, to the contrary, that there is a maximal element $\alpha\in\mathtt{P}$ such that $\dim S/\overline{\langle\mathsf{n}\setminus\mathsf{k}_\alpha\rangle}=0$. Then $S\subseteq\overline{\langle\mathsf{n}\setminus\mathsf{k}_\alpha\rangle}$, and the tuple $(\mathsf{n}\setminus\mathsf{k}_\alpha,S)$ is dependent. Therefore, the subtuple $\mathsf{k}_\alpha$ consists of coloops, and the tuple $(\mathsf{n},S)$ contains a circuit as a proper subtuple.
		
		$(\Leftarrow)$ Assume the condition holds. Supposing for a contradiction that $(\mathsf{n}, S)$ is not a circuit, there must exist coloops. This means there is a proper BK-subtuple $\mathsf{k}$ such that $S\subseteq\langle\mathsf{k}\rangle$. Choosing an element $\alpha$ corresponding to a maximal element in $\mathtt{P}_{\mathsf{n}/\mathsf{k}}=\mathtt{P}\setminus\mathtt{P}_\mathsf{k}$, yields the contradiction: $\dim S/\overline{\langle\mathsf{n}\setminus\mathsf{k}_\alpha\rangle}\leq\dim S/\overline{\langle\mathsf{n}\setminus\mathsf{k}\rangle}=0$.
	\end{proof}
	
	\begin{lem}\label{Lemma. Birkhoff poset of a BK-quotient}
		Let $\mathsf{n}$ be a BK-tuple with Birkhoff poset $\mathtt{P}$, and let $\mathsf{k}$ be a BK-subtuple of $\mathsf{n}$ with Birkhoff poset $\mathtt{P}_\mathsf{k}$. Then the Birkhoff poset of the quotient $\mathsf{n}/\mathsf{k}$ is the complement $\mathtt{P}\setminus\mathtt{P}_\mathsf{k}$.
	\end{lem}
	
	\begin{cor}\label{Corollary. BK-tuples and intristic circuits}
		Let $\mathsf{n}$ be a BK-tuple with Birkhoff poset $\mathtt{P}$, and let $\mathsf{k}$ be a BK-subtuple of $\mathsf{n}$ that contains no maximal simple BK-subtuple of $\mathsf{n}$. Then the following statements hold for the sublattice $S=\langle\mathsf{k}\rangle\cap\langle\mathsf{n}\setminus\mathsf{k}\rangle$:
		\begin{enumerate}[label=\textup{(}\textit{\alph*}\textup{)}]
			\item $\dim S/\overline{\langle\mathsf{k}\setminus\mathsf{k}_\alpha\rangle}>0$ for every maximal element $\alpha$ of the poset $\mathtt{P}_\mathsf{k}$ of the subtuple $\mathsf{k}$;
			\item the tuple $(\mathsf{k},S)$ corresponds to a circuit.
		\end{enumerate}
	\end{cor}
	\begin{proof}
		(\textit{a}) The condition that the BK-subtuple $\mathsf{k}$ contains no maximal simple BK-subtuple of $\mathsf{n}$ is equivalent to the poset $\mathtt{P}_\mathsf{k}$ containing no maximal element of the poset $\mathtt{P}$.
		
		For a maximal element $\alpha$ of the poset $\mathtt{P}_\mathsf{k}$, the tuple $\mathsf{k}\setminus\mathsf{k}_\alpha$ is a BK-subtuple with poset $\mathtt{P}_\mathsf{k}\setminus\alpha$. Hence the quotient tuple $\mathsf{n}/(\mathsf{k}\setminus\mathsf{k}_\alpha)$ is a BK-tuple with Birkhoff poset $\mathtt{P}\setminus(\mathtt{P}_\mathsf{k}\setminus\alpha)$ by Lemma \ref{Lemma. Birkhoff poset of a BK-quotient}. Suppose, to the contrary, that $\dim S/\overline{\langle\mathsf{k}\setminus\mathsf{k}_\alpha\rangle}=0$. By the definition of $S$, this is equivalent to $\dim\overline{\langle\mathsf{k}_{\alpha}/(\mathsf{k}\setminus\mathsf{k}_{\alpha})\rangle}\cap\overline{\langle(\mathsf{n}\setminus\mathsf{k})/(\mathsf{k}\setminus\mathsf{k}_{\alpha})\rangle}=0$. Since the tuples $\mathsf{n}/(\mathsf{k}\setminus\mathsf{k}_\alpha)$ and $\mathsf{k}_\alpha/(\mathsf{k}\setminus\mathsf{k}_\alpha)$ are BK-tuples, we can use the equality of dimensions
		\[
			\underbrace{\dim\langle\mathsf{n}\setminus(\mathsf{k}\setminus\mathsf{k}_{\alpha})\rangle}_{=\mathfrak{c}(\mathsf{n}\setminus(\mathsf{k}\setminus\mathsf{k}_{\alpha}))}+
			\underbrace{\dim\overline{\langle\mathsf{k}_{\alpha}/(\mathsf{k}\setminus\mathsf{k}_{\alpha})\rangle}\cap\overline{\langle(\mathsf{n}\setminus\mathsf{k})/(\mathsf{k}\setminus\mathsf{k}_{\alpha})\rangle}}_{=0}=\underbrace{\dim\langle\mathsf{k}_{\alpha}/(\mathsf{k}\setminus\mathsf{k}_{\alpha})\rangle}_{=\mathfrak{c}(\mathsf{k}_{\alpha})}+
			\dim\langle(\mathsf{n}\setminus\mathsf{k})/(\mathsf{k}\setminus\mathsf{k}_{\alpha})\rangle
		\]
		to conclude that the tuple $(\mathsf{n}\setminus\mathsf{k})/(\mathsf{k}\setminus\mathsf{k}_\alpha)$ has dimension $\mathfrak{c}(\mathsf{n}\setminus\mathsf{k})$ and is therefore a BK-tuple. Then $\alpha$ is a maximal element of the posets $\mathtt{P}\setminus(\mathtt{P}_\mathsf{k}\setminus\alpha)$ and $\mathtt{P}$, a contradiction.
		
		(\textit{b}) Combine statement (\textit{a}) with Lemma \ref{Lemma. BK-tuple, sublattice U and circuit}.
	\end{proof}
	
	We use these results to construct partitions of tuples of finite sets. The \textit{affine span} of a finite set $A$ in a lattice $M$ is the affine sublattice $\langle A\rangle$ generated by this set. By a shift, we can always ensure that the finite set $A$ contains the origin, so that its affine span is a sublattice of $M$. Then a tuple of finite sets $\mathscr{A}=(A_{1},\dots,A_{m})$, all containing the origin, generates a tuple of sublattices, and we describe a tuple of finite sets through this tuple of sublattices.
	
	By the mixed volume $\MV_{M}(\mathscr{A})$ of a tuple of finite sets $\mathscr{A}$ in a lattice $M$ we mean the mixed volume of the convex hulls of these sets. We stress that a BK-tuple $\mathscr{A}$ has positive mixed volume in the sublattice $\langle\mathscr{A}\rangle$ by Minkowski's theorem (see \cite[Theorem 8]{khovanskii_newton_2016}). Another important result establishes the compatibility of the poset partition of a reducible BK-tuple of finite sets with the mixed volume:
	\begin{thm}\label{Theorem. BK-subtuple and mixed volume}
		For tuples $\mathscr{B}\subset\mathscr{A}$ of zero defect in a lattice $M$, the mixed volume factors as 
		\[
		\MV_{M}(\mathscr{A})=\MV_{\overline{\langle\mathscr{B}\rangle}}(\mathscr{B})\,\MV_{M/\overline{\langle\mathscr{B}\rangle}}(\mathscr{A}/\mathscr{B}).
		\]
	\end{thm}
	\begin{proof}
		See \cite[Lemma 4]{steffens_mixed_2010} for a geometric proof, \cite[Lemma 3.12]{dandrea_poisson_2015} for an analytic proof, and \cite[Theorem 1.10]{esterov_galois_2019} for an algebraic proof.
	\end{proof}
	
	\begin{cor}\label{Corollary. Unique decomposition of a BK-tuple}
		A BK-tuple $\mathscr{A}$ admits a unique partition $\mathscr{A}=\bigsqcup_{\alpha\in\mathtt{P}_{\mathscr{A}}}\mathscr{B}_{\alpha}$ such that the tuple $\hat{\mathscr{B}}_{\alpha}=\mathscr{B}_{(\alpha)}/\mathscr{B}_{(\alpha)\setminus\alpha}$ is an irreducible BK-tuple for every element $\alpha$ of the poset $\mathtt{P}_{\mathscr{A}}$.
	\end{cor}
		
	For a dependent tuple of finite sets, the independent subtuples are the independent sets of the matroid induced by the realizable polymatroid \cite{pokidkin_combinatorics_2025}. To describe the discriminant, we will need the subtuple that is the maximal cycle of the induced matroid.

	\section{Discriminants of polynomial systems}\label{Section. Discriminants-of-polynomial}
	
	For a lattice $N$, the \textit{dual lattice} is $M=N^{\vee}=\mathrm{Hom}_{\mathbb{Z}}(N,\mathbb{Z})$,
	and the \textit{algebraic torus} is $T(N)=N\otimes_{\mathbb{Z}}\mathbb{C}^{\times}$.
	The \textit{character group} is $\mathrm{Hom}_{\mathbb{Z}}(T(N),\mathbb{C}^{\times})$,
	and its elements $\chi_{a}$ are called \textit{characters} \cite{fulton_introduction_1993,cox_toric_2011}. The character group is isomorphic to the dual lattice $M$. For a tuple of finite sets $\mathscr{A}\subset M$, we have defined the space of polynomial systems $\mathbb{C}_{\mathscr{A}}$; the solutions of a polynomial system $\Phi$ are sought in the torus $T(N)$.
	
	A split short exact sequence of lattices $0\rightarrow N'\rightarrow N\overset{\pi}{\rightarrow}N''\rightarrow0$ induces split short exact sequences of the dual lattices $0\rightarrow M''\overset{i}{\rightarrow}M\rightarrow M'\rightarrow0$, of the character groups, and of the tori. The projection of lattices $\pi$ induces the \textit{pushforward} $\pi_{*}=\pi\otimes\mathbb{C}^{\times}$ of tori and the \textit{pullback} $\pi^{*}=\mathrm{Hom}(\pi_{*},\mathbb{C}^{\times})$ of character groups, $\pi^{*}(\chi_{a})=\chi_{a}\circ\pi_{*}=\chi_{i(a)}$,
	for a character $\chi_{a}$ of $T(N'')$. 
	
	A monomorphism of dual lattices $M''\overset{i}{\rightarrow}M$ yields the equality of discriminants $\pi^{*}(D_{\mathscr{A}})=D_{i(\mathscr{A})}$. In particular, an isomorphism of dual lattices $M''\overset{i}{\rightarrow}M$ induces an isomorphism of discriminants $D_{\mathscr{A}}\cong D_{i(\mathscr{A})}$. 
	
	\begin{prop}
		The discriminants $D_\mathscr{A}$ and $D_\mathscr{B}$ are isomorphic if one of the following conditions holds:
			\begin{enumerate}[label=\textup{(}\textit{\alph*}\textup{)}]
			\item $\mathscr{B}=g\mathscr{A}$ for an element $g$ of the affine general linear group $\mathsf{AGL}(n,\mathbb{Z})$;
			\item $\mathscr{B}=\lambda\mathscr{A}$ for a nonzero integer $\lambda$;
			\item there exists a one-to-one correspondence between the sets of the tuples $\mathscr{A}$ and $\mathscr{B}$ under which each set $A\in\mathscr{A}$ is a translate of the corresponding set $B\in\mathscr{B}$.
		\end{enumerate}
	\end{prop}
	
	These observations allow us to choose convenient tuples in the subsequent proofs and to use the results of Section \ref{Section. Combinatorial Review}. From now on, we assume that the linear span of the tuple $\mathscr{A}$ is the whole dual lattice, $\langle\mathscr{A}\rangle=M$, and we denote the torus $T(N)$ by $T(\mathscr{A})$.
	
	Consider a split short exact sequence of dual lattices, $0\rightarrow M''\rightarrow M\rightarrow M'\rightarrow0$.
	Being finitely generated free $\mathbb{Z}$-modules, lattices are reflexive: $N\cong\mathrm{Hom}(\mathrm{Hom}(N,\mathbb{Z}),\mathbb{Z})$. Hence the lattice $\mathrm{Hom}(M'',\mathbb{Z})$ is isomorphic to a lattice $N''$ with $M''=\mathrm{Hom}(N'',\mathbb{Z})$. Let us compute the lattice $N''$. The sublattice $M''$ of the dual lattice naturally corresponds to the sublattice $N'=M''^{\perp}=\{n\in N\mid m(n)=0\;\forall m\in M''\}\subseteq N$. By the splitting, the lattice $N''$ can be taken to be the quotient $N/N'=N/M''^{\perp}$.
	
	For a subtuple $\mathscr{B}$ of a tuple $\mathscr{A}$, we look for the solutions of a subsystem from $\mathbb{C}_{\mathscr{B}}$ in the torus $T(N/\overline{\langle\mathscr{B}\rangle}^{\perp})$, which we denote by $T(\mathscr{B})$.
	
	A \textit{cofiltration} $G_{\bullet}N$ is a sequence of quotients $N=G_{i}N\twoheadrightarrow G_{i-1}N\twoheadrightarrow\dots\twoheadrightarrow G_{0}N\twoheadrightarrow0$. Cofiltrations of a lattice $N$ are in bijection with filtrations of the dual lattice $M$.
	
	Every BK-tuple $\mathscr{A}$ admits a maximal BK-filtration $F_{\bullet}$. The saturated linear spans of the tuples of this BK-filtration form a filtration $F_{\bullet}M$ of the dual lattice, $F_{i}M=\overline{\langle F_{i}\mathscr{A}\rangle}$. This gives a cofiltration $G_{\bullet}N$ and hence a cofiltration $G_{\bullet}T(N)=T(G_{\bullet}N)$ of the torus. Therefore, every polynomial system $\Phi\in\mathbb{C}_{\mathscr{A}}$ admits a BK-filtration $F_{\bullet}\Phi$, and the set of its solutions admits a cofiltration $G_{\bullet}V(\Phi)$. If the system $\Phi$ is generic, then the projections of finite sets $G_{i}V(\Phi)\twoheadrightarrow G_{i-1}V(\Phi)$ are finite covers of degree $\MV_{G_{i}N/G_{i-1}N}(F_{i}\mathscr{A}/F_{i-1}\mathscr{A})$ by Theorem \ref{Theorem. BK-subtuple and mixed volume} and the Kouchnirenko--Bernshtein theorem. For every subsystem $F_{i}(\Phi_{\mathscr{A}})=\Phi_{\mathscr{B}}$, we look for solutions in the torus $T(\mathscr{B})=T(G_{i}N)$.

	\section{BK-multiplication}\label{Section. BK-multiplicaiton}
	To compute the discriminants of BK-tuples, we introduce a specific multiplication of varieties. 
	
	Consider a split short exact sequence of dual lattices $0\rightarrow M''\rightarrow M\overset{\tau}{\rightarrow}M'\rightarrow0$,
	a finite set $A\subset M$, its projection $B=\tau(A)$, and the corresponding splitting $T(N)=T(N'')\times T(N')$ of the algebraic torus. A point of the torus $T(N)$ is a pair $(x,y)\in T(N'')\times T(N')$. The substitution $f(x,y)\overset{\ev_{x_{0}}}{\longrightarrow}f(x_{0},y)$
	of a fixed point $x_{0}$ of the torus $T(N'')$ into polynomials from $\mathbb{C}_{A}$ is a linear projection $\mathbb{C}_{A}\overset{\ev_{x_{0}}}{\longrightarrow}\mathbb{C}_{B}$ onto $\mathbb{C}_{B}$. For a tuple of finite sets $\mathscr{A}\subset M$ and $\mathscr{B}=\tau(\mathscr{A})$, the substitution $\Phi(x,y)\overset{\ev_{x_{0}}}{\longrightarrow}\Phi(x_{0},y)$
	of the point $x_{0}$ is a linear projection $\mathbb{C}_{\mathscr{A}}\overset{\ev_{x_{0}}}{\longrightarrow}\mathbb{C}_{\mathscr{B}}$.
	
	\begin{lem}\label{Lemma. Evaluation bundle is a trivial vector bundle}
		For a quasi-affine algebraic set $Y\subset\mathbb{C}_{\mathscr{B}}$, a point $x$ of the torus $T(N'')$, and the preimage $E_{\mathscr{A}}^{x}(Y)=\ev_{x}^{-1}(Y)\subseteq\mathbb{C}_{\mathscr{A}}$, the triple $(E_{\mathscr{A}}^{x}(Y),\ev_{x},Y)$ is a trivial vector bundle over $Y$ of rank $\sum_{A\in\mathscr{A}}(|A|-|\tau(A)|)$.
	\end{lem}
	
	\begin{defn}
		We call the total space $E_{\mathscr{A}}^{x}(Y)$ the \textit{evaluation bundle} over $Y$ with respect to the evaluation at the point $x$ of the torus $T(N'')$.
	\end{defn}
	
	\begin{rem}\label{Remark. Evaluation bundle}
		For tuples of finite sets $\mathscr{B}\subset\mathscr{A}$ and a point $x$ of the torus $T(\mathscr{B})$, the evaluation bundle $E_{\mathscr{A}\setminus\mathscr{B}}^{x}(Y)$ is a variety in $\mathbb{C}_{\mathscr{A}\setminus\mathscr{B}}$
		if and only if $Y$ is a variety in $\mathbb{C}_{\mathscr{A}/\mathscr{B}}$. 
	\end{rem}
	
	\begin{lem}\label{Lemma. Fiber change}
		For a chain of BK-tuples $\mathscr{C}\subset\mathscr{B}\subset\mathscr{A}$, a quasi-affine algebraic set $Y\subset\mathbb{C}_{\mathscr{B}/\mathscr{C}}$, and a point $x\in T(\mathscr{C})$, we have
		\[
		E_{\mathscr{A}\setminus\mathscr{C}}^{x}\left(Y\times\mathbb{C}_{(\mathscr{A}\setminus\mathscr{B})/\mathscr{C}}\right)=E_{\mathscr{B}\setminus\mathscr{C}}^{x}(Y)\times\mathbb{C}_{\mathscr{A}\setminus\mathscr{B}}.
		\]
	\end{lem}
	\begin{proof}
		The evaluation map $\ev_{x}\colon\mathbb{C}_{\mathscr{A}\setminus\mathscr{C}}\rightarrow\mathbb{C}_{\mathscr{A}/\mathscr{C}}$
		splits as the direct sum $\ev_{x}=u_{x}\oplus v_{x}$ of the two maps
		$u_{x}\colon\mathbb{C}_{\mathscr{B}\setminus\mathscr{C}}\rightarrow\mathbb{C}_{\mathscr{B}/\mathscr{C}}$ and $v_{x}\colon\mathbb{C}_{\mathscr{A}\setminus\mathscr{B}}\rightarrow\mathbb{C}_{(\mathscr{A}\setminus\mathscr{B})/\mathscr{C}}$.
		Hence $\ev_{x}^{-1}(Y\times\mathbb{C}_{(\mathscr{A}\setminus\mathscr{B})/\mathscr{C}})=u_{x}^{-1}(Y)\times v_{x}^{-1}(\mathbb{C}_{(\mathscr{A}\setminus\mathscr{B})/\mathscr{C}})$.
	\end{proof}
	
	\begin{thm}[Kouchnirenko--Bernshtein \cite{bernshtein_number_1975}] \label{Theorem. Reformulation of Kouchnirenko--Bernshtein theorem}
		For a tuple $\mathscr{A}$ of $n$ finite sets in an $n$-dimensional lattice $M$, there exists a nonempty open subset $U\subseteq\mathbb{C}_{\mathscr{A}}$ such that the set of solutions of every polynomial system $\Phi\in U$ consists of exactly $\MV_{M}(\mathscr{A})$ points.
	\end{thm}
	
	Following Esterov \cite{esterov_discriminant_2013}, we call the complement $B_{\mathscr{A}}=\mathbb{C}_{\mathscr{A}}\setminus U$ the \textit{bifurcation divisor}. The discriminants of all types lie in the bifurcation divisor.
	
	\begin{defn}\label{Definition. BK-multiplication}
		For BK-tuples of finite sets $\mathscr{B}\subset\mathscr{A}$, the \textit{BK-multiplication} of algebraic sets $X\subset\mathbb{C}_{\mathscr{B}}$ and $Y\subset\mathbb{C}_{\mathscr{A}/\mathscr{B}}$
		is the constructible set 
		\[
		X\circ Y=\Bigl\{\Phi\times\bigcup_{x\in V(\Phi)}\ev_{x}^{-1}(Y)\Bigm|\Phi\in X\Bigr\}\subseteq\mathbb{C}_{\mathscr{A}},
		\]
		where $V(\Phi)$ is the set of zeros of the polynomial system $\Phi\in X$ in the torus $T(\mathscr{B})$, and the preimage $\ev_{x}^{-1}(Y)=E_{\mathscr{A}\setminus\mathscr{B}}^{x}(Y)$
		is the evaluation bundle in $\mathbb{C}_{\mathscr{A}\setminus\mathscr{B}}$ over $Y$ with respect to the evaluation at a root $x\in T(\mathscr{B})$ of the polynomial system $\Phi$. We denote by $X\bullet Y$ the Zariski closure of the BK-multiplication $X\circ Y$.
	\end{defn}
	
	\begin{lem}\label{Lemma. BK-multiplication is a constructible set}
		The BK-multiplication of algebraic sets is a constructible set.
	\end{lem}
	\begin{proof}
		The BK-multiplication is the projection to the space of polynomial systems $\mathbb{C}_\mathscr{A}$ of the following quasi-affine algebraic set:
		$\{(x,\Phi_\mathscr{B},\Phi_{\mathscr{A}\setminus\mathscr{B}})\in T(\mathscr{B})\times X\times \mathbb{C}_{\mathscr{A}\setminus\mathscr{B}}\mid \Phi_\mathscr{B}(x)=0,\;\ev_x(\Phi_{\mathscr{A}\setminus\mathscr{B}})\in Y\}$.
	\end{proof}
	
	\begin{rem}
		By the Kouchnirenko--Bernshtein theorem (Theorem \ref{Theorem. Reformulation of Kouchnirenko--Bernshtein theorem}),
		the generic fiber of the BK-multiplication over $X$ is a union of $\MV_{\overline{\langle\mathscr{B}\rangle}}(\mathscr{B})$ distinct trivial evaluation bundles over $Y$.
	\end{rem}
	
	\begin{rem}
		Denote by $Z_{\mathscr{B}}$ the set of polynomial systems in $\mathbb{C}_{\mathscr{B}}$ without solutions. The contribution $\Phi\times\bigcup_{x\in V(\Phi)}\ev_{x}^{-1}(Y)$ is empty for every polynomial system $\Phi\in Z_{\mathscr{B}}$; this is why we take the Zariski closure of the multiplication.
	\end{rem}
	
	\begin{lem}\label{Lemma. BK-multiplicaiton for XoC}
		A variety $X\subseteq\mathbb{C}_{\mathscr{B}}\setminus Z_{\mathscr{B}}$ satisfies $X\circ\mathbb{C}_{\mathscr{A}/\mathscr{B}}=X\times\mathbb{C}_{\mathscr{A}\setminus\mathscr{B}}$.
	\end{lem}
	\begin{proof}
		Every system $\Phi\in X$ has at least one solution, and every fiber is the same vector space $\bigcup_{x\in V(\Phi)}\ev_{x}^{-1}(\mathbb{C}_{\mathscr{A}/\mathscr{B}})=\mathbb{C}_{\mathscr{A}\setminus\mathscr{B}}$. Hence the BK-multiplication is the product $X\times\mathbb{C}_{\mathscr{A}\setminus\mathscr{B}}$.
	\end{proof}
	
	\begin{rem}
		In particular, we have the equality $\mathbb{C}_{\mathscr{A}}=\mathbb{C}_{\mathscr{B}}\bullet\mathbb{C}_{\mathscr{A}/\mathscr{B}}$, which is a shadow of the factorization of the mixed volume in Theorem \ref{Theorem. BK-subtuple and mixed volume}.
	\end{rem}
	
	\begin{lem}\label{Lemma. The set of solutions is a variety}
		For a tuple $\mathscr{A}$, the set of solutions $V(\Phi_{\mathscr{A}}(x))$ is a variety in $T(\mathscr{A})\times\mathbb{C}_{\mathscr{A}}$.
	\end{lem}
	\begin{proof}
		Let $\Pi=\bigcap_{A\in\mathscr{A}}\Pi_{A}$ be the intersection of the hyperplanes $\Pi_{A}$ in $\mathbb{C}_{\mathscr{A}}$ defined by the equations $\sum_{a\in A}c_{a}=0$, $A\in\mathscr{A}$. Consider the map $T(\mathscr{A})\times\mathbb{C}_{\mathscr{A}}\overset{\sigma_{\mathscr{A}}}{\twoheadrightarrow}\mathbb{C}_{\mathscr{A}}$
		such that $\sigma_{\mathscr{A}}(x,c)=(c_{a}x^{a})_{a\in A\in\mathscr{A}}$. The preimage $W=\sigma_{\mathscr{A}}^{-1}(\Pi)$ is a variety isomorphic to the product $T(\mathscr{A})\times\Pi$. Indeed, the isomorphism is given by the regular maps $W\stackrel[\iota]{\sigma}{\rightleftarrows}T(\mathscr{A})\times\Pi$ such that $\iota(c_{a})=(c_{a}x^{-a})$ and $\sigma(c_{a})=(c_{a}x^{a})$ for every $a\in A\in\mathscr{A}$ and $x\in T(\mathscr{A})$ ($\iota$ and $\sigma$ are the identity on the remaining coordinates). Moreover, the preimage $\sigma_{\mathscr{A}}^{-1}(\Pi_{A})$ of each hyperplane is the zero locus of the equation $f_{A}(x)=0$ in $T(\mathscr{A})\times\mathbb{C}_{\mathscr{A}}$. Hence the variety $W$ coincides with the set of solutions $V(\Phi_{\mathscr{A}}(x))$.
	\end{proof}
	
	\begin{thm}\label{Theorem. BK-multiplication of varieties}
		For a variety $Y\subset\mathbb{C}_{\mathscr{A}/\mathscr{B}}$,
		the BK-multiplication $\mathbb{C}_{\mathscr{B}}\bullet Y$ is a variety.
	\end{thm}
	\begin{proof}
		By Lemma \ref{Lemma. The set of solutions is a variety}, for a BK-subtuple $\mathscr{B}$ the set of solutions $V(\Phi_{\mathscr{B}}(x))$ is a variety $W$ in $T(\mathscr{B})\times\mathbb{C}_{\mathscr{B}}$. Consider the evaluation map 
		\[
		T(\mathscr{B})\times\mathbb{C}_{\mathscr{A}}\overset{\ev}{\twoheadrightarrow}T(\mathscr{B})\times\mathbb{C}_{\mathscr{B}}\times\mathbb{C}_{\mathscr{A}/\mathscr{B}}
		\]
		such that $\ev(x,\Phi_{\mathscr{B}},\Phi_{\mathscr{A}\setminus\mathscr{B}})=(x,\Phi_{\mathscr{B}},\ev_{x}(\Phi_{\mathscr{A}\setminus\mathscr{B}}))$. The product $W\times Y$ is a variety, and the preimage $\ev^{-1}(W\times Y)$ is a trivial vector bundle over $W\times Y$. Indeed, for fixed $x\in T(\mathscr{B})$, the preimage is defined by the linear equations $\sum_{a\in A_{b}}c_{a}x^{a}=c_{b}$, $b\in B\in\mathscr{A}/\mathscr{B}$ (every set $A$ in $\mathscr{A}\setminus\mathscr{B}$ admits the partition $A=\bigsqcup_{b\in B}A_{b}$, where $B\in\mathscr{A}/\mathscr{B}$ is the corresponding set of the quotient). Hence the projection of the variety $\ev^{-1}(W\times Y)$ to the space $\mathbb{C}_{\mathscr{A}}$ is an irreducible constructible set, which equals the BK-multiplication $\mathbb{C}_{\mathscr{B}}\circ Y$, and its closure is a variety.
	\end{proof}
	
	\begin{lem}
		For quasi-affine algebraic sets $X,X'\subseteq\mathbb{C}_{\mathscr{B}}\setminus Z_\mathscr{B}$, $Y,Y'\subseteq\mathbb{C}_{\mathscr{A}/\mathscr{B}}$, the BK-multiplication is distributive in the following sense:
		\begin{enumerate}[label=\textup{(}\textit{\alph*}\textup{)}]
			\item $X\circ(Y\cup Y')=X\circ Y\cup X\circ Y'$,
			\item $(X\cup X')\circ Y=X\circ Y\cup X'\circ Y$,
			\item $(X\cap X')\circ Y=X\circ Y\cap X'\circ Y$.
		\end{enumerate}
	\end{lem}
	\begin{proof}
		(\textit{a}) By construction, both sides equal
		$\{\Phi\times\bigcup_{x\in V(\Phi)}\ev_{x}^{-1}(Y\cup Y')\mid \Phi\in X\}$, because $\ev_x^{-1}(Y\cup Y')=\ev_x^{-1}(Y)\cup\ev_x^{-1}(Y')$. (\textit{b}) The proof is the same as in (\textit{a}). (\textit{c}) Both sides equal $\{\Phi\times\bigcup_{x\in V(\Phi)}\ev_{x}^{-1}(Y)\mid \Phi\in X\cap X'\}$.
	\end{proof}
	
	\begin{rem}\label{Remark. BK-multiplications and intersections}
		We have given the explicit proof of this simple lemma in order to point out that, for varieties $X\subseteq\mathbb{C}_{\mathscr{B}}\setminus Z_\mathscr{B}$ and $Y,Y'\subseteq\mathbb{C}_{\mathscr{A}/\mathscr{B}}$, the inclusion $X\circ(Y\cap Y')\subset X\circ Y\cap X\circ Y'$ is usually strict, and the right-hand side is in general reducible (Remark \ref{Remark. Example of non-associativity}). The reason is that for distinct roots $x$, $x'$ of a polynomial system $\Phi\in X$, the intersection $\ev_x^{-1}(Y)\cap\ev_{x'}^{-1}(Y')$ can be nontrivial, whereas the left-hand side is determined by $\ev_x^{-1}(Y\cap Y')$ for a single root $x$. Another reason is that a resultant of resultants is reducible in general (see \cite[Proposition 3.3]{buse_explicit_2009}).
	\end{rem}
	
	If the complement $\mathscr{A}\setminus\mathscr{B}$ is a BK-tuple, then $\mathscr{A}\setminus\mathscr{B}=\mathscr{A}/\mathscr{B}$ and $\mathbb{C}_{\mathscr{A}\setminus\mathscr{B}}=\mathbb{C}_{\mathscr{A}/\mathscr{B}}$.
	
	\begin{lem}
		If the complement $\mathscr{A}\setminus\mathscr{B}$ is a BK-tuple, then the closed BK-multiplication is commutative: $X\bullet Y\cong Y\bullet X\cong \overline{X\times Y}$ for quasi-affine algebraic sets $X\subseteq\mathbb{C}_\mathscr{B}\setminus Z_\mathscr{B}$ and $Y\subseteq\mathbb{C}_{\mathscr{A}\setminus\mathscr{B}}\setminus Z_{\mathscr{A}\setminus\mathscr{B}}$.
	\end{lem}
	\begin{proof}
		In this case, the torus splits as $T(\mathscr{A})=T(\mathscr{B})\times T(\mathscr{A}\setminus\mathscr{B})$, and the evaluations $\ev_x(\Phi_{\mathscr{A}\setminus\mathscr{B}})=\Phi_{\mathscr{A}\setminus\mathscr{B}}$ and $\ev_y(\Phi_\mathscr{B})=\Phi_\mathscr{B}$ are trivial, where $x$ is a root of a polynomial system $\Phi_\mathscr{B}\in X$ and $y$ is a root of a polynomial system $\Phi_{\mathscr{A}\setminus\mathscr{B}}\in Y$. Hence $X\bullet Y$ is the closure of the projection of $\{(x,\Phi_\mathscr{B},\Phi_{\mathscr{A}\setminus\mathscr{B}})\in T(\mathscr{B})\times X\times Y\mid \Phi_\mathscr{B}(x)=0\}$ to the space $\mathbb{C}_\mathscr{A}$, which equals $\overline{X\times Y}$. Similarly, it is the closure of the projection $Y\circ X$ of $\{(y,\Phi_\mathscr{B},\Phi_{\mathscr{A}\setminus\mathscr{B}})\in T(\mathscr{A}\setminus\mathscr{B})\times X\times Y\mid \Phi_{\mathscr{A}\setminus\mathscr{B}}(y)=0\}$ to the space $\mathbb{C}_\mathscr{A}$. 
	\end{proof}
	
	\begin{prop}[Associativity]\label{Proposition. Associativity}
		For a chain of BK-tuples $\mathscr{C}\subset\mathscr{B}\subset\mathscr{A}$
		and a quasi-affine algebraic set $Y_{\mathscr{B}/\mathscr{C}}\subseteq\mathbb{C}_{\mathscr{B}/\mathscr{C}}\setminus Z_{\mathscr{B}/\mathscr{C}}$, the BK-multiplication is associative in the following sense:
		\[
			(\mathbb{C}_{\mathscr{C}}\bullet Y_{\mathscr{B}/\mathscr{C}})\bullet\mathbb{C}_{\mathscr{A}/\mathscr{B}}=\mathbb{C}_{\mathscr{C}}\bullet (Y_{\mathscr{B}/\mathscr{C}}\bullet\mathbb{C}_{\mathscr{A}/\mathscr{B}}).
		\]
	\end{prop}
	\begin{proof}
		Let us show that both algebraic sets are defined by the same equations. By Lemma \ref{Lemma. BK-multiplicaiton for XoC}, the left-hand side is the product $(\mathbb{C}_{\mathscr{C}}\bullet Y_{\mathscr{B}/\mathscr{C}})\times\mathbb{C}_{\mathscr{A}\setminus\mathscr{B}}$. As in Theorem \ref{Theorem. BK-multiplication of varieties}, this product is the closure of the projection of the quasi-affine algebraic set $(T(\mathscr{C})\times\mathbb{C}_{\mathscr{C}}\bullet Y_{\mathscr{B}/\mathscr{C}})\times\mathbb{C}_{\mathscr{A}\setminus\mathscr{B}}$ defined by the equations $\Phi_\mathscr{C}(x)=0$ and $\ev_x(\Phi_{\mathscr{B}\setminus\mathscr{C}})\in Y_{\mathscr{B}/\mathscr{C}}$, $x\in T(\mathscr{C})$.
		
		By Lemma \ref{Lemma. BK-multiplicaiton for XoC}, the right-hand side has the form $\mathbb{C}_{\mathscr{C}}\bullet (Y_{\mathscr{B}/\mathscr{C}}\times\mathbb{C}_{(\mathscr{A}\setminus\mathscr{B})/\mathscr{C}})$. This product is the closure of the projection of the quasi-affine algebraic set $T(\mathscr{C})\times\mathbb{C}_{\mathscr{C}}\bullet (Y_{\mathscr{B}/\mathscr{C}}\times\mathbb{C}_{(\mathscr{A}\setminus\mathscr{B})/\mathscr{C}})$ defined by the equations $\Phi_\mathscr{C}(x)=0$ and $\ev_x(\Phi_{\mathscr{A}\setminus\mathscr{C}})\in Y_{\mathscr{B}/\mathscr{C}}\times\mathbb{C}_{(\mathscr{A}\setminus\mathscr{B})/\mathscr{C}}$, $x\in T(\mathscr{C})$. Since the conditions $\ev_x(\Phi_{\mathscr{A}\setminus\mathscr{B}})\in\mathbb{C}_{(\mathscr{A}\setminus\mathscr{B})/\mathscr{C}}$ are vacuous, both sides are the closure of the projection of the same quasi-affine algebraic set in $T(\mathscr{C})\times{\mathbb{C}_\mathscr{A}}$.
	\end{proof}
	
	\begin{rem}\label{Remark. Example of non-associativity}
		In general, BK-multiplication is not associative, and irreducibility can fail.
		
		Consider the eleven-dimensional space of polynomial systems $f_1(x)=a_0+a_1 x+a_2 x^2$, $f_2(x,y)=b_0+b_1 y+b_2 y^2+b_3 xy$, $f_3(x,z)=c_0+c_1 x+c_2 x^2+c_3 xz$. For the chain of BK-subtuples $\mathscr{C}=(A_1)\subset\mathscr{B}=(A_1,A_2)\subset\mathscr{A}=(A_1,A_2,A_3)$, the discriminants $D_{\mathscr{B}/\mathscr{C}}$ and $D_{\mathscr{A}/\mathscr{B}}$ are irreducible, and the inclusion $\mathbb{C}_{\mathscr{C}}\bullet (D_{\mathscr{B}/\mathscr{C}}\bullet D_{\mathscr{A}/\mathscr{B}})\subset
		(\mathbb{C}_{\mathscr{C}}\bullet D_{\mathscr{B}/\mathscr{C}})\bullet D_{\mathscr{A}/\mathscr{B}}$ is strict; this can be checked at the point corresponding to the system $2-3x+x^2=1+y+y^2+xy=1+z^2+xz=0$. The right-hand side is reducible. However, by the preceding proposition, nonassociativity does not arise in this paper. The Claude AI model constructed this example.
	\end{rem}

	\section{$\mathscr{A}$-discriminants}\label{Section. Discriminants for BK-tuples}
	The idea of decomposing a polynomial system into simpler subsystems is not new \cite{sturmfels_solving_2002,buse_explicit_2009,khovanskii_newton_2016,brysiewicz_solving_2021,dickenstein_iterated_2023}; here we apply it to $\mathscr{A}$-discriminants.
	
	\begin{lem}
		\label{Lemma. Linear Algebra}In a vector space $V$, consider a linearly dependent set of vectors $\{v_{i}\}_{i\in I}$ and a linearly independent subset $\{v_{j}\}_{j\in J}$, $J\subset I$. Then the set $\{\pi(v_{k})\}_{k\in I\setminus J}$ is linearly dependent in $U$, where $\pi\colon V\rightarrow U=V/\langle v_{j}\rangle_{j\in J}$ is the projection. 
	\end{lem}
	\begin{proof}
		Let $W$ be the kernel of $\pi$, so that $V\cong U\oplus W$. Since the set $\{v_{i}\}_{i\in I}$ is linearly dependent, there exist constants $c_{i}$, not all zero, such that $\sum_{i\in I}c_{i}v_{i}=0$. Using the isomorphism, decompose each vector as $v_{i}=u_{i}+w_{i}$. The vectors $v_{j}$, $j\in J$, generate the subspace $W$, so $u_{j}=0$ for $j\in J$. Therefore, the projected set is linearly dependent: $\sum_{k\in I\setminus J}c_{k}u_{k}=0$, $u_{k}=\pi(v_{k})$.
	\end{proof}
	
	\begin{thm}\label{Theorem. Discriminatn's splitting}
		For BK-tuples $\mathscr{B}\subset\mathscr{A}$, the discriminant $D_{\mathscr{A}}$ is the union 
		\[
		D_{\mathscr{A}}=D_{\mathscr{B}}\bullet\mathbb{C}_{\mathscr{A}/\mathscr{B}}\cup\mathbb{C}_{\mathscr{B}}\bullet D_{\mathscr{A}/\mathscr{B}}.
		\]
	\end{thm}
	\begin{proof}
		For a generic system $\Phi\in D_{\mathscr{A}}$, there exists a degenerate root $(x,y)\in T(\mathscr{A})$, $x\in T(\mathscr{B})$, such that the vectors $\{df_{A}(x,y)\}_{A\in\mathscr{A}}$ are linearly dependent. If the set $\{df_{B}(x)\}_{B\in\mathscr{B}}$
		is linearly dependent, then the point $\Phi$ belongs to $D_{\mathscr{B}}\bullet\mathbb{C}_{\mathscr{A}/\mathscr{B}}$.
		Otherwise, the set $\{d(\ev_{x}f_{C})(y)\}_{C\in\mathscr{A}/\mathscr{B}}$
		is linearly dependent by Lemma \ref{Lemma. Linear Algebra}, and substituting the root $x$ into the system $\Phi_{\mathscr{A}\setminus\mathscr{B}}$ gives a system $\ev_{x}\Phi_{\mathscr{A}\setminus\mathscr{B}}$ lying in the discriminant $D_{\mathscr{A}/\mathscr{B}}$. Hence the system $\Phi$ lies in $\mathbb{C}_{\mathscr{B}}\bullet D_{\mathscr{A}/\mathscr{B}}$.
	\end{proof}
	The theorem above splits the discriminant $D_{\mathscr{A}}$ into two parts, $D_{\mathscr{B}}\bullet\mathbb{C}_{\mathscr{A}/\mathscr{B}}$ and $\mathbb{C}_{\mathscr{B}}\bullet D_{\mathscr{A}/\mathscr{B}}$. Each part may consist of several components, and the theorem says nothing about inclusions between the parts; it only shows that there may be more than one component. Nevertheless, the result is significant, since it allows us to associate with every BK-subtuple $\mathscr{B}$ a part $C(\mathscr{B})$ of the discriminant $D_{\mathscr{A}}$ by a sequence of splittings, using the associativity of BK-multiplication. An irreducible part is called a \textit{stratum}.
	
	\begin{prop}\label{Proposition. The second strata doesn't lie in the first}
		The part $\mathbb{C}_{\mathscr{B}}\bullet D_{\mathscr{A}/\mathscr{B}}$ is not contained in the part $D_{\mathscr{B}}\bullet\mathbb{C}_{\mathscr{A}/\mathscr{B}}$.
	\end{prop}
	\begin{proof}
		The part $\mathbb{C}_{\mathscr{B}}\bullet D_{\mathscr{A}/\mathscr{B}}$
		contains the dense subset $(\mathbb{C}_{\mathscr{B}}\setminus B_{\mathscr{B}})\circ D_{\mathscr{A}/\mathscr{B}}$,
		which is not contained in the closed BK-multiplication $D_{\mathscr{B}}\bullet\mathbb{C}_{\mathscr{A}/\mathscr{B}}$.
	\end{proof}
	
	\begin{cor}\label{Corollary. Component of a poset}
		Suppose that the poset $\mathtt{P}_{\mathscr{A}}$ (see Corollary \ref{Corollary. Unique decomposition of a BK-tuple}) has more than one connected component, and let $\mathtt{I}$ be one of them. Then neither of the parts $C(\mathscr{B}_{\mathtt{I}})$ and $C(\mathscr{A}\setminus\mathscr{B}_{\mathtt{I}})$ is contained in the other.
	\end{cor}
	\begin{proof}
		Note that $\mathtt{I}$ and $\mathtt{P}_{\mathscr{A}}\setminus\mathtt{I}$ are order ideals corresponding to BK-subtuples, and apply Proposition \ref{Proposition. The second strata doesn't lie in the first} twice.
	\end{proof}
	\begin{lem}[Separation lemma] \label{Lemma. Separation Lemma}
		For BK-tuples $\mathscr{C}\subset\mathscr{B}\subset\mathscr{A}$, the corresponding part $C(\mathscr{B}\setminus\mathscr{C})$ of the discriminant $D_{\mathscr{A}}$ is $\mathbb{C}_{\mathscr{C}}\bullet D_{\mathscr{B}/\mathscr{C}}\bullet\mathbb{C}_{\mathscr{A}/\mathscr{B}}$.
	\end{lem}
	
	\begin{proof}
		This follows by a direct computation using Theorem \ref{Theorem. Discriminatn's splitting} and Proposition \ref{Proposition. Associativity}:
		\begin{align*}
			D_{\mathscr{A}} & =D_{\mathscr{C}}\bullet\mathbb{C}_{\mathscr{A}/\mathscr{C}}\cup\mathbb{C}_{\mathscr{C}}\bullet(D_{\mathscr{B}/\mathscr{C}}\bullet\mathbb{C}_{(\mathscr{A}/\mathscr{C})/(\mathscr{B}/\mathscr{C})}\cup\mathbb{C}_{\mathscr{B}/\mathscr{C}}\bullet D_{(\mathscr{A}/\mathscr{C})/(\mathscr{B}/\mathscr{C})})\\
			& =D_{\mathscr{C}}\bullet\mathbb{C}_{\mathscr{A}/\mathscr{C}}\cup\mathbb{C}_{\mathscr{C}}\bullet D_{\mathscr{B}/\mathscr{C}}\bullet\mathbb{C}_{\mathscr{A}/\mathscr{B}}\cup\mathbb{C}_{\mathscr{B}}\bullet D_{\mathscr{A}/\mathscr{B}}.
		\end{align*}
	\end{proof}
	\begin{cor}\label{Corollary. Subtuple from the decomposition corresponds to irreducible stratum}
		Every element of the poset $\mathtt{P}_{\mathscr{A}}$ corresponds to a stratum of $D_{\mathscr{A}}$.
	\end{cor}
	
	\begin{proof}
		Corollary \ref{Corollary. Unique decomposition of a BK-tuple} provides the unique decomposition of the tuple $\mathscr{A}$, which is encoded by the poset $\mathtt{P}_{\mathscr{A}}$. Every element $\alpha$ of the poset $\mathtt{P}_{\mathscr{A}}$ corresponds to a subtuple $\mathscr{B}_{\alpha}$ such that the BK-tuple $\hat{\mathscr{B}}_{\alpha}$ is irreducible. The corresponding part of the discriminant $D_{\mathscr{A}}$ is $\mathbb{C}_{\mathscr{B}_{(\alpha)\setminus\alpha}}\bullet D_{\hat{\mathscr{B}}_{\alpha}}\bullet\mathbb{C}_{\mathscr{A}/\mathscr{B}_{(\alpha)}}$ by the separation lemma (Lemma \ref{Lemma. Separation Lemma}). This part is irreducible, and hence a stratum, by Theorems \ref{Theorem. BK-multiplication of varieties} and \ref{Theorem. Codimension of discriminants} on the irreducibility of the BK-multiplication and of the discriminant $D_{\hat{\mathscr{B}}_{\alpha}}$.
	\end{proof}
	Our next goal is to determine which strata are not contained in one another.
	\begin{lem} \label{Lemma. Incomparable elements}
		For incomparable elements $\alpha,\beta$ of the poset $\mathtt{P}_{\mathscr{A}}$, neither of the strata $C(\mathscr{B}_{\alpha})$ and $C(\mathscr{B}_{\beta})$ is contained in the other. 
	\end{lem}
	
	\begin{proof}
		Consider the order ideal $(\alpha)\cup(\beta)$ of the poset $\mathtt{P}_{\mathscr{A}}$, the union of the two principal order ideals, and the corresponding BK-subtuple $\mathscr{B}=\mathscr{B}_{(\alpha)\cup(\beta)}$. By Proposition \ref{Proposition. The second strata doesn't lie in the first}, it suffices to study the strata $C(\mathscr{B}_{\alpha})$ and $C(\mathscr{B}_{\beta})$ in the discriminant $D_{\mathscr{B}}$. Consider the intersection $(\alpha)\cap(\beta)$ of the principal order ideals and the BK-subtuple $\mathscr{C}=\mathscr{B}_{(\alpha)\cap(\beta)}$. By Theorem \ref{Theorem. Discriminatn's splitting}, we may further pass to the tuple $\mathscr{B}/\mathscr{C}$ and the discriminant $D_{\mathscr{B}/\mathscr{C}}$. The elements $\alpha,\beta$ are not connected in the poset $(\alpha)\cup(\beta)\setminus((\alpha)\cap(\beta))$. Therefore, the corresponding strata $C(\mathscr{B}_{\alpha}/\mathscr{C})$ and $C(\mathscr{B}_{\beta}/\mathscr{C})$ do not contain one another by Corollary \ref{Corollary. Component of a poset}. Consequently, neither of the strata $C(\mathscr{B}_{\alpha})$ and $C(\mathscr{B}_{\beta})$ is contained in the other.
	\end{proof}
	\begin{defn}
		The \textit{length} $\ell(C)$ of a chain $C$ in a poset is one less than the number of its elements. The \textit{height} $h(\alpha)$ of an element $\alpha$ of a poset is the maximal length of a chain contained in the order ideal $(\alpha)$.
		
		For comparable elements $\alpha\leq\beta$, the \textit{closed interval} $[\alpha,\beta]$ is the intersection of the principal order ideal $(\beta)$ with the principal order filter $[\alpha]$.
	\end{defn}
	
	\begin{thm}[Height theorem]\label{Theorem. Height Theorem}
		If $h(\alpha)<h(\beta)$ for elements $\alpha,\beta$ of the poset $\mathtt{P}_{\mathscr{A}}$, then the stratum $C(\mathscr{B}_{\beta})$ is not contained in the stratum
		$C(\mathscr{B}_{\alpha})$.
	\end{thm}
	\begin{proof}
		For incomparable elements $\alpha,\beta$, we use Lemma \ref{Lemma. Incomparable elements}. For comparable elements $\alpha<\beta$, we reduce to the closed interval $[\alpha,\beta]$ and the tuple $\tilde{\mathscr{A}}=\mathscr{B}_{(\beta)}/\mathscr{B}_{(\beta)\setminus[\alpha,\beta]}$ by the separation lemma (Lemma \ref{Lemma. Separation Lemma}). Then we apply the computation of the separation lemma (Lemma \ref{Lemma. Separation Lemma}) to the tuple $\mathscr{A}=\tilde{\mathscr{A}}$ and its subtuples
		$\mathscr{C}=\tilde{\mathscr{B}}_{\alpha}$, $\mathscr{B}=\tilde{\mathscr{B}}_{[\alpha,\beta]\setminus\beta}$, and use Proposition \ref{Proposition. The second strata doesn't lie in the first}. 
	\end{proof}
	
	\begin{lem}\label{Lemma. Nir subtuples}
		If the BK-tuple $\hat{\mathscr{B}}_{\alpha}=\mathscr{B}_{(\alpha)}/\mathscr{B}_{(\alpha)\setminus\alpha}$ is a nir for some element $\alpha$ of the poset $\mathtt{P}_{\mathscr{A}}$, then the stratum $C(\mathscr{B}_{\alpha})$ is a hypersurface component of the discriminant $D_{\mathscr{A}}$.
	\end{lem}
	\begin{proof}
		By the height theorem (Theorem \ref{Theorem. Height Theorem}) and the separation lemma (Lemma \ref{Lemma. Separation Lemma}), we may assume that the element $\alpha$ has height zero. The corresponding stratum is $C(\mathscr{B}_{\alpha})=D_{\mathscr{B}_{\alpha}}\bullet\mathbb{C}_{\mathscr{A}/\mathscr{B}_{\alpha}}$,
		and it has codimension one by Theorem \ref{Theorem. Codimension of discriminants}. The stratum $C(\mathscr{B}_{\alpha})$ is disjoint from $(\mathbb{C}_{\mathscr{B}_{\alpha}}\setminus B_{\mathscr{B}_{\alpha}})\circ D_{\mathscr{A}/\mathscr{B}_{\alpha}}$, which is a dense open subset of $C(\mathscr{B}_{\mathtt{P}_{\mathscr{A}}\setminus\alpha})$, a set of codimension at least one in $\mathbb{C}_{\mathscr{A}}$. Therefore, the intersection $C(\mathscr{B}_{\alpha})\cap C(\mathscr{B}_{\mathtt{P}_{\mathscr{A}}\setminus\alpha})$
		has codimension at least two, and $C(\mathscr{B}_{\alpha})\not\subset C(\mathscr{B}_{\mathtt{P}_{\mathscr{A}}\setminus\alpha})$.
	\end{proof}
	For a reducible BK-tuple $\mathscr{A}$ and an element $\alpha$ of
	the poset $\mathtt{P}_{\mathscr{A}}$, the simple BK-subtuple $\mathscr{B}_{(\alpha)}$ is prelinear if and only if the tuple $\hat{\mathscr{B}}_{\alpha}=\mathscr{B}_{(\alpha)}/\mathscr{B}_{(\alpha)\setminus\alpha}$ is linear.
	
	\begin{lem}\label{Lemma. Linear don't lie in each other}
		If the BK-tuple $\hat{\mathscr{B}}_{\beta}$ is a lir for every element $\beta$ of the principal order filter $[\alpha]$, then the stratum $C(\mathscr{B}_{\alpha})$ is a component of codimension two of the discriminant $D_{\mathscr{A}}$.
	\end{lem}
	\begin{proof}
		The complement $\mathtt{P}_{\mathscr{A}}\setminus[\alpha]$ of the principal order filter $[\alpha]$ is an order ideal. By the separation lemma (Lemma \ref{Lemma. Separation Lemma}) and Theorem \ref{Theorem. Discriminatn's splitting}, it suffices to study the discriminant of the tuple $\tilde{\mathscr{A}}=\mathscr{A}/\mathscr{B}_{\mathtt{P}_{\mathscr{A}}\setminus[\alpha]}$ whose poset is $[\alpha]$. Hence, without loss of generality, we assume that $\mathtt{P}_{\mathscr{A}}=[\alpha]$. By the height theorem (Theorem \ref{Theorem. Height Theorem}), $C(\mathscr{B}_{\alpha})\not\supset C(\mathscr{B}_{\beta})$ for every $\beta>\alpha$. 
		
		By the separation lemma (Lemma \ref{Lemma. Separation Lemma}), the corresponding stratum is $C(\mathscr{B}_{\alpha})=D_{\mathscr{B}_{\alpha}}\bullet\mathbb{C}_{\mathscr{A}/\mathscr{B}_{\alpha}}$, and it has codimension two by Theorem \ref{Theorem. Codimension of discriminants}. The stratum $C(\mathscr{B}_{\alpha})$ is disjoint from $(\mathbb{C}_{\mathscr{B}_{\alpha}}\setminus B_{\mathscr{B}_{\alpha}})\circ D_{\mathscr{A}/\mathscr{B}_{\alpha}}$, which is a dense open subset of $C(\mathscr{B}_{[\alpha]\setminus\alpha})$, a set of codimension at least two in $\mathbb{C}_{\mathscr{A}}$. Therefore, the intersection $C(\mathscr{B}_{\alpha})\cap C(\mathscr{B}_{[\alpha]\setminus\alpha})$ has codimension at least three, and $C(\mathscr{B}_{\alpha})\not\subset C(\mathscr{B}_{\mathtt{P}_{\mathscr{A}}\setminus\alpha})$.
	\end{proof}
	
	\begin{lem}\label{Lemma. Rational paramentization for the set of soltuions}
		Let $\mathscr{A}$ be a BK-tuple with poset $\mathtt{P}_\mathscr{A}$ such that the BK-tuple $\hat{\mathscr{B}}_{\alpha}$ is a lir for every $\alpha\in\mathtt{P}_\mathscr{A}$. Then the set of degenerate roots of a generic polynomial system of the discriminant $D_\mathscr{A}$ is a rational curve in the torus $T(\mathscr{A})$.
	\end{lem}
	\begin{proof}
		Recall that the set of solutions of a generic system of the discriminant $D_\mathscr{A}$ is one-dimensional by Corollary \ref{Corollary. Generic set of solutions for a system of discriminant}. If the tuple $\mathscr{A}$ is linear, this set is an affine subspace. If the tuple $\mathscr{A}$ is nonlinear, then the polynomial systems $\Phi_\mathscr{A}$ of $\mathbb{C}_{\mathscr{A}}$ have the form
		\[
			M_{\alpha}(x_\beta,\;\beta<\alpha) x_\alpha+c_\alpha (x_\beta,\;\beta<\alpha)=0,\qquad\alpha\in\mathtt{P}_\mathscr{A},
		\]
		where $x_\alpha=x_{\hat{\mathscr{B}}_\alpha}$ is the vector of variables of the torus $T(\hat{\mathscr{B}}_\alpha)$, $M_{\alpha}(x_\beta,\;\beta<\alpha)$ is a square matrix of size $\mathfrak{c}(\hat{\mathscr{B}}_\alpha)$ whose entries are Laurent polynomials in the $x_\beta$, $\beta<\alpha$, and $c_\alpha(x_\beta,\;\beta<\alpha)$ is a vector of polynomials in the $x_\beta$. For each $\alpha$, the system is linear in $x_\alpha$, and a general system $\Phi_\mathscr{A}$ in $\mathbb{C}_\mathscr{A}$ has a unique solution. This solution is a rational function of the coefficients of $\Phi_\mathscr{A}$, and it can be found by choosing a linear extension of the poset $\mathtt{P}_\mathscr{A}$ and solving the system in that order.
		
		In this case, the discriminant $D_\mathscr{A}$ is reducible and has only components of codimension two. A generic system of $D_\mathscr{A}$ corresponds to a generic system of some $D_{\hat{\mathscr{B}}_\alpha}$, whose set of solutions is a one-dimensional affine subspace. A parametrization of this subspace provides a rational parametrization of the set of solutions of the whole system $\Phi_\mathscr{A}$.
	\end{proof}
	
	Recall that the BK-subtuples corresponding to principal order ideals are the simple BK-subtuples.
	
	\begin{lem}\label{Lemma. If the upper filer consists of linear tuples then strata codim2 lies in codim1}
		Let $\mathscr{B}$ be a simple BK-tuple with maximal proper BK-subtuple $\mathscr{C}$ such that $\mathscr{B}/\mathscr{C}$ is a nir and the BK-tuple $\hat{\mathscr{B}}_{\alpha}$ is a lir for every $\alpha\in\mathtt{P}_\mathscr{C}$. Then the strata satisfy the inclusion $C(\mathscr{C})\subset C(\mathscr{B}\setminus\mathscr{C})$.
	\end{lem}
	\begin{proof}
		Since the tuple $\mathscr{B}$ is simple, the intersection of sublattices $\overline{\langle\mathscr{C}\rangle}\cap\overline{\langle\mathscr{B}\setminus\mathscr{C}\rangle}$ has positive dimension. This means that the polynomial systems $\Phi_{\mathscr{B}\setminus\mathscr{C}}$ depend nontrivially on the variables $x\in T(\mathscr{C})$.
		
		Since the quotient tuple $\mathscr{B}/\mathscr{C}$ is a nir, the discriminant $D_{\mathscr{B}/\mathscr{C}}$ is defined, up to multiplicity, by a single irreducible polynomial $P(\Phi_{\mathscr{B}/\mathscr{C}})$ by Theorem \ref{Theorem. Codimension of discriminants}. By Remark \ref{Remark. Evaluation bundle}, the evaluation bundle $E_{\mathscr{B}\setminus\mathscr{C}}^{x}(D_{\mathscr{B}/\mathscr{C}})$ is also irreducible; it is defined by the nonconstant polynomial $P_{x}(\Phi_{\mathscr{B}\setminus\mathscr{C}})$ obtained from $P(\Phi_{\mathscr{B}/\mathscr{C}})$ by replacing each coefficient $c_{b}$ with the polynomial $\ev_{x}^{-1}(c_{b})$, for every $b\in C\in\mathscr{B}/\mathscr{C}$ and a fixed point $x\in T(\mathscr{C})$.
		
		A generic point $(\Phi_{\mathscr{C}},\Phi_{\mathscr{B}\setminus\mathscr{C}})\in D_{\mathscr{C}}\times\mathbb{C}_{\mathscr{B}\setminus\mathscr{C}}$ determines in the torus $T(\mathscr{C})$ a generic hypersurface, given by the equation $P_{x}(\Phi_{\mathscr{B}\setminus\mathscr{C}})=0$, and a generic rational curve, the set of solutions of $\Phi_\mathscr{B}$, by Lemma \ref{Lemma. Rational paramentization for the set of soltuions}. Their nontrivial intersections correspond to the evaluations sending the system $\Phi_{\mathscr{B}\setminus\mathscr{C}}$
		to the discriminant $D_{\mathscr{B}/\mathscr{C}}$. Therefore, $(\Phi_{\mathscr{C}},\Phi_{\mathscr{B}\setminus\mathscr{C}})\in\mathbb{C}_{\mathscr{C}}\circ D_{\mathscr{B}/\mathscr{C}}$, and we obtain the inclusion of closures $D_{\mathscr{C}}\bullet\mathbb{C}_{\mathscr{B}/\mathscr{C}}\subset\mathbb{C}_{\mathscr{C}}\bullet D_{\mathscr{B}/\mathscr{C}}$.
	\end{proof}
	
	\begin{thm}\label{Theorem. Discriminant for a BK-tuple}
		For a BK-tuple $\mathscr{A}$, the discriminant $D_{\mathscr{A}}$ is the union $D_{\mathscr{A}}=\bigcup_{\alpha\in\mathtt{P}_{\mathscr{A}}}C(\mathscr{B}_{\alpha})$ of the strata indexed by the elements of the poset $\mathtt{P}_{\mathscr{A}}$. More precisely, the discriminant $D_{\mathscr{A}}$ is a union of components of codimension one or two: every hypersurface component is a stratum $C(\mathscr{B}_{\alpha})$ such that the tuple $\hat{\mathscr{B}}_{\alpha}$ is a nir, and every component of codimension two is a stratum $C(\mathscr{B}_{\alpha})$ such that the tuple $\hat{\mathscr{B}}_{\beta}$ is a lir for every element $\beta$ of the order filter $[\alpha]$.
	\end{thm}
	\begin{proof}
		By Corollary \ref{Corollary. Unique decomposition of a BK-tuple}, the BK-tuple $\mathscr{A}$ admits a partition $\mathscr{A}=\bigsqcup_{\alpha\in\mathtt{P}_{\mathscr{A}}}\mathscr{B}_{\alpha}$ such that every tuple $\hat{\mathscr{B}}_{\alpha}$ is an irreducible BK-tuple and every subtuple $\mathscr{B}_{\mathtt{I}}=\bigsqcup_{\alpha\in\mathtt{I}}\mathscr{B}_{\alpha}$ is a BK-tuple, for every order ideal $\mathtt{I}\subseteq\mathtt{P}_{\mathscr{A}}$. 
		
		By Corollary \ref{Corollary. Subtuple from the decomposition corresponds to irreducible stratum}, every element $\alpha$ of the poset $\mathtt{P}_{\mathscr{A}}$ corresponds to a stratum $C(\mathscr{B}_{\alpha})$. By the separation lemma (Lemma \ref{Lemma. Separation Lemma}) and Definition \ref{Definition. BK-multiplication} of the BK-multiplication, the codimension of the stratum $C(\mathscr{B}_{\alpha})$ equals the codimension of the discriminant $D_{\hat{\mathscr{B}}_{\alpha}}$
		in $\mathbb{C}_{\hat{\mathscr{B}}_{\alpha}}$. By Theorem \ref{Theorem. Codimension of discriminants},
		the codimension of $D_{\hat{\mathscr{B}}_{\alpha}}$ equals one if the tuple $\hat{\mathscr{B}}_{\alpha}$ is a nir and two if it is a lir.
		
		If an element $\alpha$ with $\hat{\mathscr{B}}_{\alpha}$ a lir is covered by an element $\beta$ with $\hat{\mathscr{B}}_{\beta}$ a nir, then the strata satisfy $C(\mathscr{B}_{\alpha})\subset C(\mathscr{B}_{\beta})$ by Lemma \ref{Lemma. If the upper filer consists of linear tuples then strata codim2 lies in codim1}. Otherwise, the strata do not contain one another and are components by Lemmas \ref{Lemma. Nir subtuples} and \ref{Lemma. Linear don't lie in each other}.
	\end{proof}
	
	\begin{rem}
		Denote by $\mathtt{Q}_{\mathscr{A}}$ the order ideal of the poset $\mathtt{P}_{\mathscr{A}}$ generated by the elements $\alpha$ for which $\hat{\mathscr{B}}_{\alpha}$ is a nir, and by $\mathscr{N}=\bigsqcup_{\alpha\in\mathtt{Q}_{\mathscr{A}}}\mathscr{B}_{\alpha}$ the corresponding BK-subtuple. By Theorem \ref{Theorem. Discriminant for a BK-tuple}, the discriminant $D_{\mathscr{N}}$ is a union of hypersurfaces, and the discriminant $D_{\mathscr{A}/\mathscr{N}}$ is a union of components of codimension two in the space $\mathbb{C}_{\mathscr{A}/\mathscr{N}}$.
	\end{rem}
	
	We now turn to the $\mathscr{A}$-discriminants of dependent tuples $\mathscr{A}$. Recall that in a matroid a \textit{circuit} is a minimal dependent subset, and a \textit{cycle} is a union of circuits.
	\begin{thm}[\cite{pokidkin_combinatorics_2025}]\label{Theorem. A/M is linearly independent}
		For a dependent tuple $\mathscr{A}$, the following are equivalent:
		\begin{enumerate}[label=\textup{(}\textit{\alph*}\textup{)}]
			\item $\mathscr{M}$ is the subtuple of minimal defect that is minimal with respect to inclusion;
			\item $\mathscr{M}$ is the maximal essential dependent subtuple;
			\item $\mathscr{M}$ is the maximal cycle of the matroid induced by the polymatroid on $\mathscr{A}$.
		\end{enumerate}
		The subtuple $\mathscr{M}$ is unique, and the quotient $\mathscr{A}/\mathscr{M}$ is independent.
	\end{thm}
	
	Theorem \ref{Theorem. A/M is linearly independent} provides a new characterization of the subtuple $\mathscr{M}$ underlying the resultant \cite{sturmfels_newton_1994}: it is the maximal cycle of the matroid induced by the realizable polymatroid.
	
	\begin{thm}\label{Theprem. Codimesnion of a resultant}
		For a dependent tuple $\mathscr{A}$ with maximal essential subtuple $\mathscr{M}$, the sparse resultant $R_{\mathscr{A}}$ equals the sparse resultant $R_{\mathscr{M}}$, which has codimension $-\delta(\mathscr{M})$.
	\end{thm}
	\begin{proof}
		See \cite[Theorem 1.1]{sturmfels_newton_1994} and \cite[Theorem 2.15]{esterov_determinantal_2007} for algebro-geometric proofs, and  \cite[Theorem 2.23]{jensen_computing_2013} for a tropical proof.
	\end{proof}
	
	\begin{cor}\label{Corollary. Discriminants for linearly dependent tuples}
		For a dependent tuple $\mathscr{A}$ with maximal essential subtuple $\mathscr{M}$, the discriminant $D_{\mathscr{A}}$ is the sparse resultant $R_{\mathscr{M}}$.
	\end{cor}
	\begin{proof}
		Every root of a polynomial system with a dependent support tuple $\mathscr{A}$ is degenerate. Hence the discriminant $D_\mathscr{A}$ equals the resultant $R_\mathscr{M}$.
	\end{proof}

	\section{Cayley discriminants}\label{Section. Cayley discriminants}
	
	The \textit{Cayley discriminant} $D_{\cay(\mathscr{A})}$ is the discriminant of the Cayley set $\cay(\mathscr{A})$. If a BK-tuple $\mathscr{A}$ is irreducible, then the $\mathscr{A}$-discriminant coincides with the Cayley discriminant (Corollary \ref{Corollary. Irreducible BK => three types of discriminant are equal}). The vector spaces $\mathbb{C}_{\mathscr{A}}$ and $\mathbb{C}_{\cay(\mathscr{A})}$ are isomorphic.
		
	\begin{lem}\label{Lemma. Cayley discriminant for a connected component of the poset}
		Let $\mathscr{B}\subset\mathscr{A}$ be BK-tuples such that the poset $\mathtt{P}_{\mathscr{B}}$ is a connected component of the poset $\mathtt{P}_{\mathscr{A}}$. Then the Cayley discriminant is the product $D_{\cay(\mathscr{A})}\cong D_{\cay(\mathscr{B})}\times D_{\cay(\mathscr{A}\setminus\mathscr{B})}$.
	\end{lem}
	\begin{proof}
		Since the poset $\mathtt{P}_{\mathscr{B}}$ is a connected component, the tuple $\mathscr{A}\setminus\mathscr{B}$ is a BK-tuple. Hence the linear span of the BK-tuple $\mathscr{A}$ decomposes into the direct sum $\langle\mathscr{A}\rangle=\langle\mathscr{B}\rangle\oplus\langle\mathscr{A}\setminus\mathscr{B}\rangle$, and the torus splits as $T(\mathscr{A})=T(\mathscr{B})\times T(\mathscr{A}\setminus\mathscr{B})$. Consequently, every system $\Phi_{\mathscr{A}}(x,y)$ in $\mathbb{C}_{\mathscr{A}}$ splits into two independent systems $\Phi_{\mathscr{B}}(x)$ and $\Phi_{\mathscr{A}\setminus\mathscr{B}}(y)$, and every polynomial $f\in\mathbb{C}_{\cay(\mathscr{A})}$ can be written as
		\[
		f(x,y,\lambda)=\sum_{A\in\mathscr{A}}\lambda_{A}f_{A}(x,y)=\sum_{B\in\mathscr{B}}\lambda_{B}f_{B}(x)+\sum_{C\in\mathscr{A}\setminus\mathscr{B}}\lambda_{C}f_{C}(y).
		\]
		This means that a point $(x^{\circ},y^{\circ},\lambda^{\circ})$ is a singular point of the polynomial $f_{\cay(\mathscr{A})}$ if and only if $(x^{\circ},\lambda_{\mathscr{B}}^{\circ})$ and $(y^{\circ},\lambda_{\mathscr{A}\setminus \mathscr{B}}^{\circ})$ are singular points of the polynomials $f_{\cay(\mathscr{B})}$ and $f_{\cay(\mathscr{A}\setminus\mathscr{B})}$, respectively. 
	\end{proof}
	
	Consider a dependent tuple $\mathscr{A}$ whose bases are BK-tuples. We say that a polynomial system $\Phi_\mathscr{A}\in\mathbb{C}_\mathscr{A}$ and its root $x\in T(\mathscr{A})$ are in \textit{general position} if the determinant $\det(df_B(x))_{B\in\mathscr{B}}$ is nonzero for every basis $\mathscr{B}$ of the tuple $\mathscr{A}$.
	
	\begin{lem}\label{Lemma. Non-zero constants of a linear combination for a circuit}
		Let $\mathscr{A}$ be a tuple that is a circuit of the induced matroid. Then, for a generic polynomial system $\Phi_\mathscr{A}\in R_\mathscr{A}$ and its root $x$, all the coefficients $\lambda_A$ of a nontrivial linear relation $\sum_{A\in\mathscr{A}}\lambda_A df_A(x)=0$ are nonzero, and the root $x$ is a nondegenerate multiple root.
	\end{lem}
	\begin{proof}
		If the tuple $\mathscr{A}$ is a circuit, then the subtuple $\mathscr{A}\setminus A$ is a BK-tuple for every set $A\in\mathscr{A}$. Hence, for a polynomial system $\Phi_\mathscr{A}$ and its root $x$ in general position, the determinants $\det(df_B(x))_{B\in\mathscr{A}\setminus A}$ are nonzero, and no nontrivial linear combination $\sum_{B\in\mathscr{A}\setminus A}\lambda_B df_B(x)$ vanishes, for every set $A\in\mathscr{A}$. Therefore, in general position none of the coefficients $\lambda_A$, $A\in\mathscr{A}$, can be zero. In particular, the root $x$ is a nondegenerate multiple root by definition.
	\end{proof}
	
	\begin{prop}\label{Proposition. Simplification of a Cayley discriminant}
		For a BK-tuple $\mathscr{A}$ with poset $\mathtt{P}_{\mathscr{A}}$ and a BK-subtuple $\mathscr{B}$ that contains no maximal (with respect to inclusion) simple BK-subtuple of $\mathscr{A}$, the Cayley discriminant satisfies $D_{\cay(\mathscr{A})}\cong\mathbb{C}_{\mathscr{B}}\bullet D_{\cay(\mathscr{A}/\mathscr{B})}$.
	\end{prop}
	\begin{proof}
		We can choose a basis such that the systems in $\mathbb{C}_{\mathscr{B}}$ depend on the variables $x\in T(\mathscr{B})$ and the systems in $\mathbb{C}_{\mathscr{A}\setminus\mathscr{B}}$ depend on the variables $(x,y)\in T(\mathscr{A})$. Then every polynomial $f\in\mathbb{C}_{\cay(\mathscr{A})}$ has the form $f(x,y,\lambda)=\sum_{B\in\mathscr{B}}\lambda_{B}f_{B}(x)+\sum_{C\in\mathscr{A}\setminus\mathscr{B}}\lambda_{C}f_{C}(x,y)$.
		
		$(\subseteq)$ A generic polynomial $f(x,y,\lambda)$ in the discriminant $D_{\cay(\mathscr{A})}$ has a singular point $(x^{\circ},y^{\circ},\lambda^{\circ})$ in the torus $T(\cay(\mathscr{A}))$ such that
		\begin{align*}
			\nabla_{y}f(x^{\circ},y^{\circ},\lambda^{\circ}) & =\sum_{C\in\mathscr{A}\setminus\mathscr{B}}\lambda_{C}^{\circ}\,\nabla_{y}f_{C}(x^{\circ},y^{\circ})=\sum_{C\in\mathscr{A}\setminus\mathscr{B}}\lambda_{C}^{\circ}\,\nabla_{y}(\ev_{x^{\circ}}f_{C})(y^{\circ})=0.
		\end{align*}
		This means that the point $(y^{\circ},\lambda_{\mathscr{A}\setminus\mathscr{B}}^{\circ})$ is a singular point of $f_{\cay(\mathscr{A}/\mathscr{B})}=\sum_{C\in\mathscr{A}\setminus\mathscr{B}}\lambda_{C}^{\circ}\,(\ev_{x^{\circ}}f_{C})(y)$, and the corresponding system $(\Phi_{\mathscr{B}},f_{\cay(\mathscr{A}\setminus\mathscr{B})})$ belongs to $\mathbb{C}_{\mathscr{B}}\bullet D_{\cay(\mathscr{A}/\mathscr{B})}$.
		
		$(\supseteq)$ For a generic system $(\Phi_{\mathscr{B}},f_{\cay(\mathscr{A}\setminus\mathscr{B})})$ in $\mathbb{C}_{\mathscr{B}}\circ D_{\cay(\mathscr{A}/\mathscr{B})}$, there exists a root $(x^{\circ},y^{\circ},\lambda_{\mathscr{A}\setminus\mathscr{B}}^{\circ})$ in the torus $T(\mathscr{B})\times T(\cay(\mathscr{A}/\mathscr{B}))$ such that $\sum_{C\in\mathscr{A}\setminus\mathscr{B}}\lambda_{C}^{\circ}\,\nabla_{y}(\ev_{x^{\circ}}f_{C})(y^{\circ})=0$. 
		
		Denote by $A$ the support of the polynomial $\sum_{C\in\mathscr{A}\setminus\mathscr{B}}\lambda_{C}^{\circ}\,(\ev_{y^{\circ}}f_{C})(x)$. Since the subtuple $\mathscr{B}$ contains no maximal simple BK-subtuple, the tuple $(\mathscr{B}, A)$ is a circuit by Corollary \ref{Corollary. BK-tuples and intristic circuits}. In general position, all the coefficients $\{\mu_{B}\}_{B\in\mathscr{B}}$ and $\mu$ of a nontrivial linear relation 
		\[
		\sum_{B\in\mathscr{B}}\mu_{B}\,\nabla_{x}f_{B}(x^{\circ})+\mu\nabla_{x}f_A(x^{\circ})=0
		\]
		are nonzero by Lemma \ref{Lemma. Non-zero constants of a linear combination for a circuit}. Therefore, the corresponding Cayley polynomial $f_{\cay(\mathscr{A})}$ has the singular point $(x^{\circ},y^{\circ},\mu_{\mathscr{B}},\mu\lambda_{\mathscr{A}\setminus\mathscr{B}}^{\circ})$ in the torus $T(\cay(\mathscr{A}))$ and hence belongs to the discriminant $D_{\cay(\mathscr{A})}$.
	\end{proof}
	Some components of $\mathscr{A}$-discriminants are known to be Cayley discriminants.
	\begin{thm}[Esterov {\cite[Theorem 2.31]{esterov_newton_2010}}] \label{Theorem. Esterov, hypersurface components}
		For a reducible BK-tuple $\mathscr{A}$, the components of codimension one of the $\mathscr{A}$-discriminant are Cayley discriminants of BK-subtuples.
	\end{thm}
	We extend this result to the components of codimension two of $\mathscr{A}$-discriminants.
	
	\begin{prop}\label{Proposition. Discriminant's strata are Cayley discriminants}
		For an element $\alpha$ of the poset $\mathtt{P}_{\mathscr{A}}$, the stratum $C(\mathscr{B}_{\alpha})$ of the $\mathscr{A}$-discriminant can be expressed through the Cayley discriminant of the BK-tuple $\mathscr{B}_{(\alpha)}$: 
		\[C(\mathscr{B}_{\alpha})=
		\mathbb{C}_{\mathscr{B}_{(\alpha)\setminus\alpha}}\bullet D_{\hat{\mathscr{B}}_{\alpha}}\bullet\mathbb{C}_{\mathscr{A}\setminus\mathscr{B}_{(\alpha)}}=
		D_{\cay(\mathscr{B}_{(\alpha)})}\times\mathbb{C}_{\mathscr{A}\setminus\mathscr{B}_{(\alpha)}}.\]
	\end{prop}
	\begin{proof}
		By the separation lemma (Lemma \ref{Lemma. Separation Lemma}) and Corollary \ref{Corollary. Subtuple from the decomposition corresponds to irreducible stratum}, the stratum corresponding to an element $\alpha$ has the form $C(\mathscr{B}_{\alpha})=
		\mathbb{C}_{\mathscr{B}_{(\alpha)\setminus\alpha}}\bullet D_{\hat{\mathscr{B}}_{\alpha}}\bullet\mathbb{C}_{\mathscr{A}\setminus\mathscr{B}_{(\alpha)}}$. By Lemma \ref{Lemma. BK-multiplicaiton for XoC}, we have $C(\mathscr{B}_{\alpha})=
		\mathbb{C}_{\mathscr{B}_{(\alpha)\setminus\alpha}}\bullet D_{\hat{\mathscr{B}}_{\alpha}}\times\mathbb{C}_{\mathscr{A}\setminus\mathscr{B}_{(\alpha)}}$, so it suffices to show that $D_{\cay(\mathscr{B}_{(\alpha)})}=\mathbb{C}_{\mathscr{B}_{(\alpha)\setminus\alpha}}\bullet D_{\hat{\mathscr{B}}_{\alpha}}$.
		
		The order ideal $(\alpha)$ has the unique maximal element $\alpha$. We apply Proposition \ref{Proposition. Simplification of a Cayley discriminant} iteratively, using the associativity of Proposition \ref{Proposition. Associativity}. At the last step, we obtain the Cayley discriminant of the irreducible BK-tuple $\hat{\mathscr{B}}_{\alpha}$, which coincides with the $\mathscr{A}$-discriminant by Theorem \ref{Theorem. Codimension of discriminants}.
	\end{proof}

	\begin{thm}\label{Theorem. Cayley Discriminants}
		For a BK-tuple $\mathscr{A}$ with poset $\mathtt{P}_{\mathscr{A}}$, the Cayley discriminant $D_{\cay(\mathscr{A})}$ equals
		\[
		D_{\cay(\mathscr{A})}=\mathbb{C}_{\mathscr{C}}\bullet\prod_{\alpha\in\max\mathtt{P}_{\mathscr{A}}}D_{\hat{\mathscr{B}}_{\alpha}},
		\]
		where $\mathscr{C}=\mathscr{A}\setminus\bigcup_{\alpha\in \max \mathtt{P}_{\mathscr{A}}}\mathscr{B}_{\alpha}$. The Cayley discriminant has codimension $\mathfrak{n}+2\mathfrak{l}$, where $\mathfrak{n}$ (respectively, $\mathfrak{l}$) is the number of maximal elements $\alpha$ of the poset $\mathtt{P}_{\mathscr{A}}$ such that the BK-tuple $\hat{\mathscr{B}}_{\alpha}$ is a nir (respectively, a lir).
	\end{thm}
	\begin{proof}
		By Proposition \ref{Proposition. Simplification of a Cayley discriminant} and the associativity of Proposition \ref{Proposition. Associativity}, the Cayley discriminant of the BK-tuple $\mathscr{A}$ reduces to the Cayley discriminant of the BK-tuple $\mathscr{A}/\mathscr{C}$,
		where $\mathscr{C}=\mathscr{A}\setminus\bigcup_{\alpha\in \max \mathtt{P}_{\mathscr{A}}}\mathscr{B}_{\alpha}$. The poset $\mathtt{P}_{\mathscr{A}/\mathscr{C}}$ is a disjoint union of $|\max \mathtt{P}_{\mathscr{A}}|$ incomparable elements. This means that the BK-tuple $\mathscr{A}/\mathscr{C}$ is a disjoint union of irreducible BK-subtuples, and $\mathbb{C}_{\mathscr{A}/\mathscr{C}}=\prod_{\alpha\in \max \mathtt{P}_{\mathscr{A}}}\mathbb{C}_{\hat{\mathscr{B}}_{\alpha}}$. Hence the Cayley discriminant $D_{\cay(\mathscr{A}/\mathscr{C})}$ equals the direct product $\prod_{\alpha\in\max \mathtt{P}_{\mathscr{A}}}D_{\hat{\mathscr{B}}_{\alpha}}$ by Lemma \ref{Lemma. Cayley discriminant for a connected component of the poset} and the equality of discriminants $D_{\cay(\hat{\mathscr{B}}_{\alpha})}=D_{\hat{\mathscr{B}}_{\alpha}}$ for irreducible BK-tuples $\hat{\mathscr{B}}_{\alpha}$.
	\end{proof}
	
	\begin{rem}
		By Remark \ref{Remark. BK-multiplications and intersections}, the inclusion $\mathbb{C}_{\mathscr{C}}\bullet\prod_{\alpha\in\max\mathtt{P}_{\mathscr{A}}}D_{\hat{\mathscr{B}}_{\alpha}}\subset
		\bigcap_{\alpha\in \max\mathtt{P}_{\mathscr{A}}}\mathbb{C}_{\mathscr{A}\setminus\mathscr{B}_{\alpha}}\bullet D_{\hat{\mathscr{B}}_{\alpha}}$ can be strict, and the right-hand side can be reducible; see the example in Remark \ref{Remark. Example of non-associativity}.
	\end{rem}
	
	For a tuple $\mathscr{A}$ and an arbitrary subtuple $\mathscr{B}$, we define a \textit{multiplication} of varieties by the same formula as in Definition \ref{Definition. BK-multiplication}. We call a tuple $\mathscr{A}$ a \textit{BK-union} if it is a union of BK-subtuples; such a tuple may be dependent.
	
	\begin{thm}\label{Theorem. Cayley disccriminant for a BK-union}
		For a BK-union $\mathscr{A}$, the Cayley discriminant $D_{\cay(\mathscr{A})}$ is the variety $R_\mathscr{M}\bullet D_{\cay(\mathscr{A}/\mathscr{M})}$ of codimension $-\delta(\mathscr{M})+(\mathfrak{n}+2\mathfrak{l})(\mathscr{A}/\mathscr{M})$, where $R_{\mathscr{M}}$ is the resultant of the maximal essential subtuple $\mathscr{M}$, and $\mathfrak{n}$ (respectively, $\mathfrak{l}$) is the number of maximal elements $\alpha$ of the poset $\mathtt{P}_{\mathscr{A}/\mathscr{M}}$ such that the BK-tuple $\hat{\mathscr{B}}_{\alpha}$ is a nir (respectively, a lir).
	\end{thm}
	\begin{proof}
		Following the proof of Proposition \ref{Proposition. Simplification of a Cayley discriminant}, one shows that for a dependent tuple $\mathscr{A}$ the Cayley discriminant equals the multiplication $R_\mathscr{M}\bullet D_{\cay(\mathscr{A}/\mathscr{M})}$. The subtuple $\mathscr{A}\setminus\mathscr{M}$ consists of coloops of the matroid induced by the polymatroid on $\mathscr{A}$, and the quotient $\mathscr{A}/\mathscr{M}$ is independent by Theorem \ref{Theorem. A/M is linearly independent} and is a BK-tuple \cite{pokidkin_combinatorics_2025}. Since the additional coefficients in the Cayley trick are not allowed to vanish, this explains the distinguished contribution of $D_{\cay(\mathscr{A}/\mathscr{M})}$ to the multiplication. The codimension of $R_\mathscr{M}$ is given by Theorem \ref{Theprem. Codimesnion of a resultant}, and Theorem \ref{Theorem. Cayley Discriminants} describes the codimension of the Cayley discriminant of $\mathscr{A}/\mathscr{M}$.
	\end{proof}
	
	\begin{cor}[{\cite[Proposition 6.1, Lemma 6.3]{dickenstein_tropical_2007}}]\label{Corollary. Cayley discriminant for an essential linearly dependent tuple}
		For an essential dependent tuple $\mathscr{A}$, the Cayley discriminant equals the sparse resultant $R_{\mathscr{A}}$.
	\end{cor}

	\section{Mixed discriminants}\label{Section. Mixed discriminants}
	\begin{lem}\label{Lemma. Mixed discriminant for a connected component of the poset}
		If the poset $\mathtt{P}_{\mathscr{A}}$ of a BK-tuple $\mathscr{A}$ has more than one connected component, then the mixed discriminant is empty.
	\end{lem}
	\begin{proof}
		Every connected component corresponds to a BK-subtuple of the BK-tuple $\mathscr{A}$, and the tuple $\mathscr{A}$ decomposes into a disjoint union of BK-subtuples. Hence the variables and the polynomial systems split into independent subsystems with BK-supports, and a polynomial system with support $\mathscr{A}$ cannot have a nondegenerate multiple root.
	\end{proof}
	For a BK-tuple $\mathscr{A}$, denote its mixed discriminant by $\mathring{D}_{\mathscr{A}}$.
	
	\begin{lem}\label{Lemma. Simplification of a mixed discriminant}
		For a BK-tuple $\mathscr{A}$ and a BK-subtuple $\mathscr{B}$ that contains no maximal simple BK-subtuple of $\mathscr{A}$, the mixed discriminant $\mathring{D}_{\mathscr{A}}$ equals $\mathbb{C}_{\mathscr{B}}\bullet\mathring{D}_{\mathscr{A}/\mathscr{B}}$.
	\end{lem}
	\begin{proof}
		$(\subseteq)$ If the mixed discriminant $\mathring{D}_{\mathscr{A}}$ is nonempty, then its generic system $\Phi_{\mathscr{A}}$ has a nondegenerate multiple root $(x,y)\in T(\mathscr{B})\times T(\mathscr{A}/\mathscr{B})$ such that the vectors $\{df_{A}(x,y)\}_{A\in\mathscr{A}}$ are linearly dependent and the vectors $\{df_{B}(x,y)\}_{B\in\mathscr{A}\setminus A}$ are independent for every set $A\in\mathscr{A}$. Then there is a nontrivial linear relation $\sum_{A\in\mathscr{A}}\lambda_A df_A(x,y)=0$ in which all the coefficients $\lambda_A$ are nonzero. Moreover, the vectors $\{df_{B}(x)\}_{B\in\mathscr{C}}$ span a vector space of dimension $\mathfrak{c}(\mathscr{C})$ for every proper subtuple $\mathscr{C}\subset\mathscr{A}$. For polynomial systems in $\mathbb{C}_\mathscr{B}$, the vector space generated by $\{df_{B}(x)\}_{B\in\mathscr{B}}$ is always the same. If the point $x$ is a degenerate root of the subsystem $\Phi_{\mathscr{B}}$, then the point $(x,y)$ cannot be a nondegenerate multiple root of the system $\Phi_\mathscr{A}$, because the vectors $\{df_{B}(x,y)\}_{B\in\mathscr{A}\setminus A}$ are dependent for the sets $A\in\mathscr{A}\setminus\mathscr{B}$. Therefore, for a nondegenerate multiple root $(x,y)$, the vectors $\{df_{B}(x)\}_{B\in\mathscr{B}}$ are independent.
		
		If the tuple $\mathscr{B}$ contained a maximal simple BK-subtuple $\mathscr{B}_{(\alpha)}$ of $\mathscr{A}$, then the element $\alpha$ would be maximal in both posets $\mathtt{P}_\mathscr{B}$ and $\mathtt{P}_\mathscr{A}$. Then the tuples $\mathscr{B}_{(\alpha)\setminus\alpha}$ and $\mathscr{A}\setminus\mathscr{B}_\alpha$ would be BK-tuples, and the polynomial system $\ev_z(\Phi_{\mathscr{B}_\alpha})$ would depend on its own set of variables $x_{\mathscr{B}_\alpha}$ for a root $z=x_{\mathscr{B}_{(\alpha)\setminus\alpha}}$. This would force the coefficients $\{\lambda_B\}_{B\in\mathscr{B}_\alpha}$ to vanish, a contradiction.
		
		If the tuple $\mathscr{B}$ contains no maximal simple BK-subtuple, then the tuple $(\mathscr{B},A)$ is a circuit by Corollary \ref{Corollary. BK-tuples and intristic circuits}, where $A$ is the support of the polynomial $\sum_{C\in\mathscr{A}\setminus\mathscr{B}}\lambda_{C}\,(\ev_{y}f_{C})(x)$. Then Lemma \ref{Lemma. Non-zero constants of a linear combination for a circuit} guarantees that the coefficients $\{\lambda_B\}_{B\in\mathscr{B}}$ are nonzero. Therefore, by Lemma \ref{Lemma. Linear Algebra}, we are reduced to the system $\ev_{x}\Phi_{\mathscr{A}\setminus\mathscr{B}}$ with linearly dependent vectors $\{d(\ev_{x}f_{C})(y)\}_{C\in\mathscr{A}/\mathscr{B}}$ and the nondegenerate multiple root $y$. This gives the inclusion of discriminants.
		 
		$(\supseteq)$ If a point $y\in T(\mathscr{A}/\mathscr{B})$ is a nondegenerate multiple root of a polynomial system $\Phi_{\mathscr{A}/\mathscr{B}}$, then the vectors $\{df_C(y)\}_{C\in(\mathscr{A}/\mathscr{B})\setminus A}$ span the same vector space of dimension $\mathfrak{c}(\mathscr{A}/\mathscr{B})-1$ for every set $A\in \mathscr{A}/\mathscr{B}$. For a generic system $\Phi_\mathscr{A}$, the vectors $\{df_B\}_{B\in\mathscr{B}}$ are independent. This means that the point $(x,y)$ is a nondegenerate multiple root of a system $\Phi_\mathscr{A}$ such that $\ev_x(\Phi_{\mathscr{A}\setminus\mathscr{B}})=\Phi_{\mathscr{A}/\mathscr{B}}$, where $x$ is a root of a generic system $\Phi_\mathscr{B}$. 
	\end{proof}
	
	\begin{thm}\label{Theorem. The mixed discriminant for a BK-tuple}
		For a BK-tuple $\mathscr{A}$, the mixed discriminant equals the Cayley discriminant if the poset $\mathtt{P}_{\mathscr{A}}$ has only one maximal element. Otherwise, the mixed discriminant is empty.
	\end{thm}
	\begin{proof}
		By Lemma \ref{Lemma. Simplification of a mixed discriminant} and the associativity of Proposition \ref{Proposition. Associativity}, the mixed discriminant of the BK-tuple $\mathscr{A}$ reduces to the mixed discriminant of the BK-tuple $\mathscr{A}/\mathscr{C}$, where $\mathscr{C}=\mathscr{A}\setminus\bigcup_{\alpha\in \max \mathtt{P}_{\mathscr{A}}}\mathscr{B}_{\alpha}$. The poset $\mathtt{P}_{\mathscr{A}/\mathscr{C}}$ is a disjoint union of $|\max \mathtt{P}_{\mathscr{A}}|$ incomparable elements. By Lemma \ref{Lemma. Mixed discriminant for a connected component of the poset}, the mixed discriminant $\mathring{D}_{\mathscr{A}/\mathscr{C}}$ is nonempty only if the poset $\mathtt{P}_{\mathscr{A}/\mathscr{C}}$ consists of a single element $\alpha$. In that case, the BK-tuple $\mathscr{A}/\mathscr{C}$ is irreducible, and the three types of discriminants coincide. Therefore, by Proposition \ref{Proposition. Discriminant's strata are Cayley discriminants} applied to the BK-tuple $\mathscr{A}$, both the mixed discriminant and the Cayley discriminant equal $\mathbb{C}_{\mathscr{C}}\bullet D_{\mathscr{A}/\mathscr{C}}$.
	\end{proof}
	
	\begin{rem}
		It is shown in \cite{cattani_mixed_2013} that if the Cayley discriminant of a BK-tuple is a hypersurface, then the mixed discriminant is the same hypersurface. Theorem \ref{Theorem. The mixed discriminant for a BK-tuple} (respectively, Theorem \ref{Theorem. The mixed discriminant for a linearly dependent tuple}) provides a second proof of this result for nonlinear BK-tuples (respectively, for dependent tuples).
	\end{rem}
	
	\begin{thm}\label{Theorem. The mixed discriminant for a linearly dependent tuple}
		The mixed discriminant of a dependent tuple $\mathscr{A}$ is nonempty if and only if the tuple $\mathscr{A}$ is a circuit. In this case, the mixed discriminant is the resultant $R_{\mathscr{A}}$, and it is a hypersurface.
	\end{thm}
	\begin{proof}
		If the tuple $\mathscr{A}$ contains more than one circuit, then the tuple $\mathscr{A}\setminus A$ is dependent for every set $A\in\mathscr{A}$. This means that the systems in $\mathbb{C}_{\mathscr{A}}$ cannot have a nondegenerate multiple root, and the mixed discriminant is empty. Hence the tuple $\mathscr{A}$ contains exactly one circuit.
		
		If the tuple $\mathscr{A}$ contains a coloop $A$, then the tuple $\mathscr{A}\setminus A$ is dependent, and again the systems in $\mathbb{C}_{\mathscr{A}}$ cannot have a nondegenerate multiple root, so the mixed discriminant is empty. Therefore, the only remaining possibility is that the tuple $\mathscr{A}$ is a circuit.
		
		If the tuple $\mathscr{A}$ is a circuit, then the tuple $\mathscr{A}\setminus A$ is independent for every set $A\in\mathscr{A}$. Then, for a generic polynomial system in $R_{\mathscr{A}}$ and its root in general position, the root is a nondegenerate multiple root by Lemma \ref{Lemma. Non-zero constants of a linear combination for a circuit}, and the mixed discriminant coincides with the resultant $R_{\mathscr{A}}$. Since circuits of the induced matroid always have defect $-1$ \cite{pokidkin_combinatorics_2025}, the mixed discriminant is a hypersurface.
	\end{proof}
	
	\section{Degrees of discriminants}\label{Section. Degrees of components}
	For the reader's convenience, we begin this section with a table summarizing the degrees of the discriminants.
	\begin{table}[htbp]
		\caption{Degrees of the three discriminants for dependent and BK-tuples}
		\label{Table. Degrees of discriminants}
		\small
		\setlength{\tabcolsep}{4pt}
		\renewcommand{\arraystretch}{1.25}
		\begin{tabular}{@{}>{\raggedright\arraybackslash}p{2.55cm}
				>{\raggedright\arraybackslash}p{3.95cm}
				>{\raggedright\arraybackslash}p{3.95cm}
				>{\raggedright\arraybackslash}p{3.95cm}@{}}
			\toprule
			Tuple $\mathscr{A}$
			& $\mathscr{A}$-discriminant
			& Cayley discriminant
			& Mixed discriminant \\
			\midrule
			% ---------- dependent ----------
			Dependent
			& Propositions \ref{Proposition. I Degree of the sparse resultant for a linearly dependent tuple},
			\ref{Proposition. II Degree of the sparse resultant for a linearly dependent tuple}
			& Theorem \ref{Theorem. Matsui--Takeuchi, Degree of the Cayley discriminant}
			& Corollary \ref{Corollary. Degree of the mixed discriminant for a linearly dependent tuple}\\
			% ---------- BK ----------
			BK-tuple
			& Corollary \ref{Corollary. Degrees for components of the discriminant for a BK-tuple}, Proposition \ref{Proposition. Degree of a lir}
			& Theorem \ref{Theorem. Matsui--Takeuchi, Degree of the Cayley discriminant}
			& Remark \ref{Remark. Degree of the mixed discriminant for a BK-tuple}
		\end{tabular}
	\end{table}
	
	\begin{cor}[\cite{pedersen_product_1993}]\label{Corollary. Degree of the mixed discriminant for a linearly dependent tuple}
		For a tuple $\mathscr{A}$ that is a circuit, the mixed discriminant has degree 
		\[\deg \mathring{D}_{\mathscr{A}}= \sum_{A\in\mathscr{A}}\MV_{\overline{\langle\mathscr{A}\rangle}}(\mathscr{A}\setminus A).\]
	\end{cor}
	
	For an essential dependent tuple $\mathscr{A}$, both the $\mathscr{A}$-discriminant and the Cayley discriminant equal the sparse resultant $R_{\mathscr{A}}$ by Corollaries \ref{Corollary. Discriminants for linearly dependent tuples} and \ref{Corollary. Cayley discriminant for an essential linearly dependent tuple}. We present two formulas for the degree of a sparse resultant. Recall that the sparse resultant of a dependent tuple equals the sparse resultant of its maximal essential subtuple.
	
	\begin{prop}\label{Proposition. I Degree of the sparse resultant for a linearly dependent tuple}
		For an essential dependent tuple $\mathscr{A}$, the sparse resultant $R_{\mathscr{A}}$ has degree $\sum_{\mathscr{B}\in\mathcal{B}(\mathscr{A})}\MV_{\overline{\langle\mathscr{A}\rangle}}(\mathscr{B})$,
		where $\mathcal{B}(\mathscr{A})$ is the set of bases of the induced matroid on $\mathscr{A}$.
	\end{prop}
	\begin{proof}
		By \cite[Corollary 6.5]{dickenstein_tropical_2007}, the degree of the sparse resultant $R_{\mathscr{A}}$ is the sum of the mixed volumes $\MV_{\overline{\langle\mathscr{A}\rangle}}(\mathscr{B})$
		over all subtuples $\mathscr{B}$ of cardinality $\dim \langle\mathscr{A}\rangle$. By \cite{pokidkin_combinatorics_2025}, the bases of the induced matroid on $\mathscr{A}$ are the BK-subtuples of cardinality $\dim \langle\mathscr{A}\rangle$. Hence every subtuple of cardinality $\dim \langle\mathscr{A}\rangle$ is either dependent or a basis, that is, a BK-tuple. Therefore, in the formula of \cite{dickenstein_tropical_2007}, only the bases of the induced matroid contribute to the degree of the sparse resultant $R_{\mathscr{A}}$.
	\end{proof}
	
	Let $\mathscr{A}$ be an essential dependent tuple of defect $-r$. Consider the lattice $S=\overline{\langle\mathscr{A}\rangle}\times\mathbb{Z}^{r}$
	and the new tuple $\mathscr{A}^{\flat}=(A\times\triangle_{r},\,A\in\mathscr{A})$, where $\triangle_{r}$ is the standard simplex in $\mathbb{Z}^{r}$ and each set $A\times\triangle_{r}$ is regarded as a subset of the lattice $S$. By construction, the tuple $\mathscr{A}^{\flat}$ has zero defect.
	
	\begin{prop}[Esterov]\label{Proposition. II Degree of the sparse resultant for a linearly dependent tuple}
		For an essential dependent tuple $\mathscr{A}$, the sparse resultant $R_{\mathscr{A}}$ has degree $\MV_{S}(\mathscr{A}^{\flat})$.
	\end{prop}
	\begin{proof}
		The degree of the resultant $R_{\mathscr{A}}$, which has codimension $r$, equals the number of its intersection points with a generic $r$-dimensional affine subspace $\Pi$ of $\mathbb{C}_{\mathscr{A}}$. Choose a parametrization of the subspace $\Pi$, $\Phi=\Phi^{0}+\sum_{i=1}^{r}y_{i}\Phi^{i}$ for some fixed points $\Phi^{0},\dots,\Phi^{r}$ of $\mathbb{C}_{\mathscr{A}}$ and new variables $y_{1},\dots,y_{r}$. Each coefficient of the system $\Phi$ is a linear function of the new variables $y$, and the parametrization yields a polynomial system with support in $\mathbb{C}_{\mathscr{A}^{\flat}}$. By the Kouchnirenko--Bernshtein theorem, a generic system in $\mathbb{C}_{\mathscr{A}^{\flat}}$ has $\MV_{S}(\mathscr{A}^{\flat})$ solutions. Each solution corresponds to an intersection point of the $r$-dimensional subspace $\Pi$ with the resultant.
		
		Since the tuple $\mathscr{A}$ is essential, the tuple $\mathscr{A}^{\flat}$ is irreducible, because every proper subtuple $\mathscr{B}^{\flat}$ has positive defect: $\delta(\mathscr{B}^{\flat})=\delta(\mathscr{B})-\delta(\mathscr{A})>0$. This ensures that the mixed volume $\MV_{S}(\mathscr{A}^{\flat})$ is positive.
	\end{proof}
	
	For a face $A'\prec A$ of a set $A$ in a lattice, consider the projection of saturated sublattices $s\colon\tilde{M}\rightarrow\tilde{M}/\tilde{M}'$, where $\tilde{M}$ and $\tilde{M}'$ are the sublattices parallel to the affine spans of $A$ and $A'$. Denote by $c^{A',A}$ the difference of the integer volumes of the sets $s(A)$ and $s(A\setminus A')$ in the lattice $s(\overline{\langle A\rangle})$. Set $c^{A,A}=1$, and $c^{A',A}=0$ if $A'$ is not a face of $A$. This defines a square matrix $C$ with entries $c^{A'',A'}$ indexed by all pairs of faces $A''$, $A'$ of $A$; the entries $e^{A'',A'}$ of the inverse matrix $C^{-1}$ are called \textit{Euler obstructions} (see \cite{esterov_newton_2010}). Denote by $\langle A\rangle$ the affine span of a set $A$.
	
	\begin{thm}[Matsui--Takeuchi {\cite[Theorem 1.1]{matsui_geometric_2011}}]\label{Theorem. Matsui--Takeuchi, Degree of the Cayley discriminant}
		For a finite set $A$ whose $A$-discriminant has codimension $i$, the degree of the $A$-discriminant is
		\[
		\deg D_{A}=\sum_{A'\prec A}e^{A',A}\,\left(\binom{\dim\langle A'\rangle-1}{i}+(-1)^{i+1}(i+1)\right)\Vol(A').
		\] 
	\end{thm}
	
	For an $A$-discriminant of codimension one, the formula agrees with the results of \cite{gelfand_discriminants_1994,dickenstein_tropical_2007}. In particular, this theorem gives the degrees of Cayley discriminants when applied to the Cayley set $\cay(\mathscr{A})$ of a tuple $\mathscr{A}$. If the tuple $\mathscr{A}$ is essential and dependent, we use the formula with codimension $i=-\delta(\mathscr{A})$; if the tuple $\mathscr{A}$ is a BK-tuple, we use it with codimension $i=2\mathfrak{l}+\mathfrak{n}$ by Theorem \ref{Theorem. Cayley Discriminants}. Moreover, for a BK-tuple $\mathscr{A}$, Theorem \ref{Theorem. Matsui--Takeuchi, Degree of the Cayley discriminant} allows us to compute the degrees of the components of the $\mathscr{A}$-discriminant. 
	
	\begin{cor}\label{Corollary. Degrees for components of the discriminant for a BK-tuple}
		For a BK-tuple $\mathscr{A}$ and $\alpha\in \mathtt{P}_{\mathscr{A}}$, the component $C(\mathscr{B}_{\alpha})$ has degree 
		\begin{align*}
			\deg C(\mathscr{B}_{\alpha}) & =\begin{cases}
				\sum\limits_{A\prec \cay(\mathscr{B}_{(\alpha)})}e^{A,\cay(\mathscr{B}_{(\alpha)})}\,(\dim \langle A\rangle+1)\Vol(A), & \text{if }\hat{\mathscr{B}}_{\alpha}\text{ is a nir,}\\
				\frac{1}{2}\sum\limits_{A\prec \cay(\mathscr{B}_{(\alpha)})}e^{A,\cay(\mathscr{B}_{(\alpha)})}\,(\dim \langle A\rangle+1)\,(\dim \langle A\rangle-4)\Vol(A), & \text{if }\hat{\mathscr{B}}_{\alpha}\text{ is a lir.}
			\end{cases}
		\end{align*}
	\end{cor}
	
	\begin{rem}\label{Remark. Degree of the mixed discriminant for a BK-tuple}
		Corollary \ref{Corollary. Degrees for components of the discriminant for a BK-tuple} also gives the degrees of the mixed discriminants of BK-tuples $\mathscr{A}$ whose poset $\mathtt{P}_{\mathscr{A}}$ has a unique maximal element $\alpha$, so that $\mathscr{A}=\mathscr{B}_{(\alpha)}$.
	\end{rem}
	\begin{rem}
		For a nir $\hat{\mathscr{B}}_{\alpha}$ and its face $A$, the
		number $(\dim \langle A\rangle+1)\Vol(A)$ is the total degree of the $A$-Euler discriminant (see \cite[Corollary 1.9]{esterov_discriminant_2013}).
	\end{rem}
	
	Every face of a Cayley set is the Cayley set of a collection of faces. Esterov expressed the volumes of Cayley sets (see \cite[Lemma 1.7]{esterov_newton_2010}) and the mixed volumes of tuples of Cayley sets (see \cite{esterov_multiplicities_2012}) through the mixed volumes of their generating sets. This computation of volumes simplifies the Matsui--Takeuchi degree formula (see \cite[Corollary 1.14]{esterov_discriminant_2013}).
	
	Denote the mixed volume of $n$ finite sets by the monomial $A_{1}\cdots A_{n}$, and the integer simplex $\{a\in\mathbb{Z}_{\geq0}^{k}\mid a_{1}+\dots+a_{k}=n\}$
	by $\triangle_{k}(n)$.
	
	\begin{cor}[Danilov--Khovanskii \cite{danilov_newton_1987}; Esterov \cite{esterov_newton_2010}]
		For a tuple of finite sets $\mathscr{A}=(A_{1},\dots,A_{k})$, the volume of the Cayley set $\cay(\mathscr{A})$ in its $n$-dimensional linear span equals
		\[
		\Vol(\cay(\mathscr{A}))=\sum_{a\in\triangle_{k}(n)}A_{1}^{a_{1}}\cdots A_{k}^{a_{k}}.
		\]
	\end{cor}
		
	\begin{prop}\label{Proposition. Degree of a lir}
		For the linear BK-tuple $\mathscr{A}$ of cardinality $n$ consisting of simplices, the $\mathscr{A}$-discriminant has degree $n(n+1)/2$.
	\end{prop}
	\begin{proof}
		For the linear BK-tuple of simplices, the discriminant $D_{\mathscr{A}}$ is a determinantal variety; see, for example, \cite[Example 19.10]{harris_algebraic_1992} for a proof.
	\end{proof}
	
	\section*{Conclusion}
	For the reader's convenience, the description of the components and their codimensions for $\mathscr{A}$-discriminants, Cayley discriminants, and mixed discriminants is collected in Table \ref{Table. Components of discriminants}. The description is almost complete. Our approach does not cover the explicit computation of the codimension of the Cayley discriminant for a dependent tuple $\mathscr{A}$ whose quotient by the maximal cycle $\mathscr{M}$ is a nontrivial independent tuple, because the tuple $\mathscr{A}/\mathscr{M}$ can then be an arbitrary independent tuple. However, there are at least four approaches to describing the codimension of a Cayley discriminant \cite{cattani_non-splitting_2022}.
	
	\begin{problem*}
		The description of the components and degrees of mixed discriminants and $\mathscr{A}$-discriminants remains open for underdetermined polynomial systems with more than one equation.
	\end{problem*}
	
	For a BK-tuple, the Birkhoff poset of its BK-subtuples coincides with the poset of the principal partition of its polymatroid. More generally, the discriminants of underdetermined polynomial systems are conjecturally described by the principal partitions of their polymatroids \cite{pokidkin_combinatorics_2025}.
	
	\section*{Acknowledgments}
	I am grateful to Alexander Esterov for his supervision and to my family for their support.
	
	\section*{AI disclosure}
	All results were obtained by the author independently, and no proofs were generated by AI. The author used Claude Fable 5.1 (Claude Max, September 2026) to identify potential errors, proof gaps, and misprints, which were then resolved by the author. Additionally, the AI model was used to improve language and clarity.
	
	\printbibliography

@book{stanley_enumerative_2011,
	edition = {2},
	title = {Enumerative {Combinatorics}},
	copyright = {https://www.cambridge.org/core/terms},
	doi = {10.1017/CBO9781139058520},
	publisher = {Cambridge University Press},
	author = {Stanley, Richard P.},
	year = {2011},
}

@book{cox_toric_2011,
	address = {Providence, Rhode Island},
	series = {Graduate {Studies} in {Mathematics}},
	title = {Toric {Varieties}},
	volume = {124},
	doi = {10.1090/gsm/124},
	publisher = {American Mathematical Society},
	author = {Cox, David and Little, John and Schenck, Henry},
	year = {2011},
}

@article{buse_explicit_2009,
	title = {Explicit factors of some iterated resultants and discriminants},
	volume = {78},
	doi = {10.1090/S0025-5718-08-02111-X},
	number = {265},
	journal = {Mathematics of Computation},
	publisher = {American Mathematical Society},
	author = {Busé, Laurent and Mourrain, Bernard},
	year = {2009},
	note = {\href{https://arxiv.org/abs/cs/0612050}{ArXiv}},
	pages = {345--386},
}

@book{fulton_introduction_1993,
	series = {Annals of {Mathematics} {Studies}},
	title = {Introduction to toric varieties},
	volume = {131},
	doi = {10.1515/9781400882526},
	publisher = {Princeton University Press},
	author = {Fulton, William},
	month = jul,
	year = {1993},
}

@article{cattani_non-splitting_2022,
	title = {Non-splitting {Flags}, {Iterated} {Circuits}, $\underline{\sigma}$-{Matrices} and {Cayley} {Configurations}},
	volume = {50},
	issn = {2305-221X, 2305-2228},
	doi = {10.1007/s10013-022-00554-7},
	number = {3},
	journal = {Vietnam Journal of Mathematics},
	author = {Cattani, Eduardo and Dickenstein, Alicia},
	month = jul,
	year = {2022},
	note = {\href{https://arxiv.org/abs/2105.00302v2}{ArXiv}},
	pages = {679--706},
}

@article{crowley_bergman_2026,
	title = {The {Bergman} fan of a polymatroid},
	volume = {379},
	copyright = {https://www.ams.org/publications/copyright-and-permissions},
	issn = {1088-6850, 0002-9947},
	doi = {10.1090/tran/9492},
	number = {5},
	journal = {Transactions of the American Mathematical Society},
	author = {Crowley, Colin and Huh, June and Larson, Matt and Simpson, Connor and Wang, Botong},
	year = {2026},
	note = {\href{https://arxiv.org/abs/2207.08764v2}{ArXiv}},
	pages = {3239--3256},
}

@incollection{dokken_plane_2014,
	address = {Cham},
	series = {Geometry and {Computing}},
	title = {Plane {Mixed} {Discriminants} and {Toric} {Jacobians}},
	volume = {10},
	doi = {10.1007/978-3-319-08635-4_6},
	booktitle = {{SAGA} – {Advances} in {ShApes}, {Geometry}, and {Algebra}},
	publisher = {Springer International Publishing},
	author = {Dickenstein, Alicia and Emiris, Ioannis Z. and Karasoulou, Anna},
	editor = {Dokken, Tor and Muntingh, Georg},
	month = jan,
	year = {2014},
	note = {\href{https://arxiv.org/abs/1304.5809v1}{ArXiv}},
	pages = {105--121},
}

@incollection{huh_likelihood_2014,
	address = {Cham},
	series = {Lecture {Notes} in {Mathematics}},
	title = {Likelihood {Geometry}},
	volume = {2108},
	doi = {10.1007/978-3-319-04870-3_3},
	booktitle = {Combinatorial {Algebraic} {Geometry}: {Levico} {Terme}, {Italy} 2013},
	publisher = {Springer International Publishing},
	author = {Huh, June and Sturmfels, Bernd},
	editor = {Di Rocco, Sandra and Sturmfels, Bernd},
	month = jan,
	year = {2014},
	note = {\href{https://arxiv.org/abs/1305.7462v2}{ArXiv}},
	pages = {63--117},
}

@article{danilov_newton_1987,
	title = {Newton polyhedra and an algorithm for computing {Hodge}-{Deligne} numbers},
	volume = {29},
	issn = {0025-5726},
	doi = {10.1070/IM1987v029n02ABEH000970},
	number = {2},
	journal = {Mathematics of the USSR-Izvestiya},
	author = {Danilov, Vladimir I. and Khovanskii, Askold G.},
	month = apr,
	year = {1987},
	pages = {279--298},
}

@book{gelfand_discriminants_1994,
	address = {Boston, MA},
	title = {Discriminants, {Resultants}, and {Multidimensional} {Determinants}},
	copyright = {http://www.springer.com/tdm},
	doi = {10.1007/978-0-8176-4771-1},
	publisher = {Birkhäuser Boston},
	author = {Gelfand, Israel M. and Kapranov, Mikhail M. and Zelevinsky, Andrei V.},
	year = {1994},
}

@article{gelfand_hypergeometric_1989,
	title = {Hypergeometric functions and toric varieties},
	volume = {23},
	copyright = {http://www.springer.com/tdm},
	issn = {0016-2663, 1573-8485},
	doi = {10.1007/BF01078777},
	number = {2},
	journal = {Functional Analysis and Its Applications},
	author = {Gelfand, Israel M. and Zelevinsky, Andrei V. and Kapranov, Mikhail M.},
	month = apr,
	year = {1989},
	pages = {94--106},
}

@article{sturmfels_newton_1994,
	title = {On the {Newton} {Polytope} of the {Resultant}},
	volume = {3},
	issn = {0925-9899},
	doi = {10.1023/A:1022497624378},
	number = {2},
	journal = {Journal of Algebraic Combinatorics},
	author = {Sturmfels, Bernd},
	month = apr,
	year = {1994},
	pages = {207--236},
}

@article{jensen_computing_2013,
	title = {Computing tropical resultants},
	volume = {387},
	copyright = {https://www.elsevier.com/tdm/userlicense/1.0/},
	issn = {0021-8693},
	doi = {10.1016/j.jalgebra.2013.03.031},
	journal = {Journal of Algebra},
	author = {Jensen, Anders and Yu, Josephine},
	month = aug,
	year = {2013},
	note = {\href{https://arxiv.org/abs/1109.2368v2}{ArXiv}},
	pages = {287--319},
}

@article{esterov_multiplicities_2012,
	title = {Multiplicities of degenerations of matrices and mixed volumes of {Cayley} polyhedra},
	volume = {6},
	issn = {1949-2006},
	doi = {10.5427/jsing.2012.6d},
	journal = {Journal of Singularities},
	author = {Esterov, Alexander},
	month = jul,
	year = {2012},
	note = {\href{https://arxiv.org/abs/1205.4344v1}{ArXiv}},
	pages = {27--36},
}

@article{curran_restriction_2007,
	title = {Restriction of {A}-{Discriminants} and {Dual} {Defect} {Toric} {Varieties}},
	volume = {42},
	issn = {0747-7171},
	doi = {10.1016/j.jsc.2006.02.006},
	number = {1-2},
	journal = {Journal of Symbolic Computation},
	author = {Curran, Raymond and Cattani, Eduardo},
	month = jan,
	year = {2007},
	note = {\href{https://arxiv.org/abs/math/0510615v2}{ArXiv}},
	pages = {115--135},
}

@article{esterov_systems_2015,
	title = {Systems of equations with a single solution},
	volume = {68},
	issn = {0747-7171},
	doi = {10.1016/j.jsc.2014.09.007},
	journal = {Journal of Symbolic Computation},
	author = {Esterov, Alexander and Gusev, Gleb},
	month = may,
	year = {2015},
	note = {\href{https://arxiv.org/abs/1211.6763v2}{ArXiv}},
	pages = {116--130},
}

@article{steffens_mixed_2010,
	title = {Mixed volume techniques for embeddings of {Laman} graphs},
	volume = {43},
	issn = {0925-7721},
	doi = {10.1016/j.comgeo.2009.04.004},
	number = {2},
	journal = {Computational Geometry},
	author = {Steffens, Reinhard and Theobald, Thorsten},
	month = feb,
	year = {2010},
	note = {\href{https://arxiv.org/abs/0805.4120v2}{ArXiv}},
	pages = {84--93},
}

@article{dandrea_poisson_2015,
	title = {A {Poisson} formula for the sparse resultant},
	volume = {110},
	copyright = {http://doi.wiley.com/10.1002/tdm\_license\_1.1},
	issn = {0024-6115},
	doi = {10.1112/plms/pdu069},
	number = {4},
	journal = {Proceedings of the London Mathematical Society},
	author = {D’Andrea, Carlos and Sombra, Martín},
	month = apr,
	year = {2015},
	note = {\href{https://arxiv.org/abs/1310.6617v2}{ArXiv}},
	pages = {932--964},
}

@incollection{vanhove_feynman_2019,
	address = {Cham},
	series = {Texts \& {Monographs} in {Symbolic} {Computation}},
	title = {Feynman {Integrals}, {Toric} {Geometry} and {Mirror} {Symmetry}},
	doi = {10.1007/978-3-030-04480-0_17},
	booktitle = {Elliptic {Integrals}, {Elliptic} {Functions} and {Modular} {Forms} in {Quantum} {Field} {Theory}},
	publisher = {Springer International Publishing},
	author = {Vanhove, Pierre},
	editor = {Blümlein, Johannes and Schneider, Carsten and Paule, Peter},
	month = jan,
	year = {2019},
	note = {\href{https://arxiv.org/abs/1807.11466v2}{ArXiv}},
	pages = {415--458},
}

@article{rado_theorem_1942,
	title = {A theorem on independence relations},
	volume = {os-13},
	issn = {0033-5606, 1464-3847},
	doi = {10.1093/qmath/os-13.1.83},
	number = {1},
	journal = {The Quarterly Journal of Mathematics},
	author = {Rado, Richard},
	year = {1942},
	pages = {83--89},
}

@article{pedersen_product_1993,
	title = {Product formulas for resultants and {Chow} forms},
	volume = {214},
	copyright = {http://www.springer.com/tdm},
	issn = {0025-5874, 1432-1823},
	doi = {10.1007/BF02572411},
	number = {1},
	journal = {Mathematische Zeitschrift},
	author = {Pedersen, Paul and Sturmfels, Bernd},
	month = sep,
	year = {1993},
	pages = {377--396},
}

@article{pagaria_hodge_2023,
	title = {Hodge {Theory} for {Polymatroids}},
	volume = {2023},
	copyright = {https://academic.oup.com/journals/pages/open\_access/funder\_policies/chorus/standard\_publication\_model},
	issn = {1073-7928, 1687-0247},
	doi = {10.1093/imrn/rnad001},
	number = {23},
	journal = {International Mathematics Research Notices},
	author = {Pagaria, Roberto and Pezzoli, Gian Marco},
	month = dec,
	year = {2023},
	note = {\href{https://arxiv.org/abs/2105.04214v2}{ArXiv}},
	pages = {20118--20168},
}

@article{matsui_geometric_2011,
	title = {A geometric degree formula for {A}-discriminants and {Euler} obstructions of toric varieties},
	volume = {226},
	issn = {0001-8708},
	doi = {10.1016/j.aim.2010.08.020},
	number = {2},
	journal = {Advances in Mathematics},
	author = {Matsui, Yutaka and Takeuchi, Kiyoshi},
	month = jan,
	year = {2011},
	note = {\href{https://arxiv.org/abs/0807.3163v5}{ArXiv}},
	pages = {2040--2064},
}

@article{khovanskii_newton_2016,
	title = {Newton polytopes and irreducible components of complete intersections},
	volume = {80},
	issn = {1064-5632, 1468-4810},
	doi = {10.1070/IM8307},
	number = {1},
	journal = {Izvestiya: Mathematics},
	author = {Khovanskii, Askold G.},
	month = feb,
	year = {2016},
	pages = {263--284},
}

@article{gelfand_generalized_1990,
	title = {Generalized {Euler} integrals and {A}-hypergeometric functions},
	volume = {84},
	copyright = {https://www.elsevier.com/tdm/userlicense/1.0/},
	issn = {0001-8708},
	doi = {10.1016/0001-8708(90)90048-R},
	number = {2},
	journal = {Advances in Mathematics},
	author = {Gelfand, Israel M. and Kapranov, Mikhail M. and Zelevinsky, Andrei V.},
	month = dec,
	year = {1990},
	pages = {255--271},
}

@article{esterov_multivariate_2016,
	title = {Multivariate {Abel}–{Ruffini}},
	volume = {365},
	issn = {0025-5831, 1432-1807},
	doi = {10.1007/s00208-015-1309-6},
	number = {3-4},
	journal = {Mathematische Annalen},
	author = {Esterov, Alexander and Gusev, Gleb},
	month = aug,
	year = {2016},
	note = {\href{https://arxiv.org/abs/1405.1252v3}{ArXiv}},
	pages = {1091--1110},
}

@article{esterov_galois_2019,
	title = {Galois theory for general systems of polynomial equations},
	volume = {155},
	issn = {0010-437X, 1570-5846},
	doi = {10.1112/S0010437X18007868},
	number = {2},
	journal = {Compositio Mathematica},
	author = {Esterov, Alexander},
	month = feb,
	year = {2019},
	note = {\href{https://arxiv.org/abs/1801.08260v3}{ArXiv}},
	pages = {229--245},
}

@article{esterov_newton_2010,
	title = {Newton {Polyhedra} of {Discriminants} of {Projections}},
	volume = {44},
	issn = {0179-5376, 1432-0444},
	doi = {10.1007/s00454-010-9242-7},
	number = {1},
	journal = {Discrete \& Computational Geometry},
	author = {Esterov, Alexander},
	month = jul,
	year = {2010},
	note = {\href{https://arxiv.org/abs/0810.4996v3}{ArXiv}},
	pages = {96--148},
}

@article{esterov_determinantal_2007,
	title = {Determinantal {Singularities} and {Newton} {Polyhedra}},
	volume = {259},
	copyright = {http://www.springer.com/tdm},
	issn = {0081-5438, 1531-8605},
	doi = {10.1134/S0081543807040037},
	number = {1},
	journal = {Proceedings of the Steklov Institute of Mathematics},
	author = {Esterov, Alexander},
	month = dec,
	year = {2007},
	note = {\href{https://arxiv.org/abs/0906.5097v2}{ArXiv}},
	pages = {16--34},
}

@article{dickenstein_tropical_2007,
	title = {Tropical discriminants},
	volume = {20},
	issn = {0894-0347, 1088-6834},
	doi = {10.1090/S0894-0347-07-00562-0},
	number = {4},
	journal = {Journal of the American Mathematical Society},
	author = {Dickenstein, Alicia and Feichtner, Eva and Sturmfels, Bernd},
	month = apr,
	year = {2007},
	note = {\href{https://arxiv.org/abs/math/0510126v3}{ArXiv}},
	pages = {1111--1133},
}

@article{dickenstein_iterated_2023,
	title = {Iterated and mixed discriminants},
	volume = {7},
	issn = {2415-6302, 2415-6310},
	doi = {10.4171/jca/68},
	number = {1},
	journal = {Journal of Combinatorial Algebra},
	author = {Dickenstein, Alicia and Di Rocco, Sandra and Morrison, Ralph},
	month = may,
	year = {2023},
	note = {\href{https://arxiv.org/abs/2101.11571v2}{ArXiv}},
	pages = {45--81},
}

@article{di_rocco_projective_2006,
	title = {Projective {Duality} of {Toric} {Manifolds} and {Defect} {Polytopes}},
	volume = {93},
	issn = {0024-6115, 1460-244X},
	doi = {10.1017/S0024611505015686},
	number = {1},
	journal = {Proceedings of the London Mathematical Society},
	author = {Di Rocco, Sandra},
	month = jul,
	year = {2006},
	note = {\href{https://arxiv.org/abs/math/0305150v2}{ArXiv}},
	pages = {85--104},
}

@article{cattani_mixed_2013,
	title = {Mixed discriminants},
	volume = {274},
	issn = {1432-1823},
	doi = {10.1007/s00209-012-1095-8},
	number = {3},
	journal = {Mathematische Zeitschrift},
	author = {Cattani, Eduardo and Cueto, María Angélica and Dickenstein, Alicia and Di Rocco, Sandra and Sturmfels, Bernd},
	month = aug,
	year = {2013},
	note = {\href{https://arxiv.org/abs/1112.1012v1}{ArXiv}},
	pages = {761--778},
}

@article{brysiewicz_solving_2021,
	title = {Solving decomposable sparse systems},
	volume = {88},
	issn = {1017-1398, 1572-9265},
	doi = {10.1007/s11075-020-01045-x},
	number = {1},
	journal = {Numerical Algorithms},
	author = {Brysiewicz, Taylor and Rodriguez, Jose Israel and Sottile, Frank and Yahl, Thomas},
	month = sep,
	year = {2021},
	note = {\href{https://arxiv.org/abs/2001.04228v2}{ArXiv}},
	pages = {453--474},
}

@article{borger_defectivity_2020,
	title = {On defectivity of families of full-dimensional point configurations},
	volume = {7},
	issn = {2330-1511},
	doi = {10.1090/bproc/46},
	number = {4},
	journal = {Proceedings of the American Mathematical Society, Series B},
	author = {Borger, Christopher and Nill, Benjamin},
	month = may,
	year = {2020},
	note = {\href{https://arxiv.org/abs/1801.07467v2}{ArXiv}},
	pages = {43--51},
}

@article{bernshtein_number_1975,
	title = {The number of roots of a system of equations},
	volume = {9},
	issn = {1573-8485},
	doi = {10.1007/BF01075595},
	number = {3},
	journal = {Functional Analysis and Its Applications},
	author = {Bernshtein, David N.},
	month = jul,
	year = {1975},
	pages = {183--185},
}

@article{amendola_maximum_2019,
	title = {The maximum likelihood degree of toric varieties},
	volume = {92},
	issn = {0747-7171},
	doi = {10.1016/j.jsc.2018.04.016},
	journal = {Journal of Symbolic Computation},
	author = {Améndola, Carlos and Bliss, Nathan and Burke, Isaac and Gibbons, Courtney R. and Helmer, Martin and Hoşten, Serkan and Nash, Evan D. and Rodriguez, Jose Israel and Smolkin, Daniel},
	month = may,
	year = {2019},
	note = {\href{https://arxiv.org/abs/1703.02251v2}{ArXiv}},
	pages = {222--242},
}

@article{esterov_discriminant_2013,
	title = {The discriminant of a system of equations},
	volume = {245},
	copyright = {https://www.elsevier.com/tdm/userlicense/1.0/},
	issn = {0001-8708},
	doi = {10.1016/j.aim.2013.06.027},
	journal = {Advances in Mathematics},
	author = {Esterov, Alexander},
	month = oct,
	year = {2013},
	note = {\href{https://arxiv.org/abs/1110.4060v2}{ArXiv}},
	pages = {534--572},
}

@book{sturmfels_solving_2002,
	series = {Conference {Board} of the {Mathematical} {Sciences} regional conference series in mathematics},
	title = {Solving {Systems} of {Polynomial} {Equations}},
	volume = {97},
	doi = {10.1090/cbms/097},
	publisher = {American Mathematical Society},
	author = {Sturmfels, Bernd},
	month = oct,
	year = {2002},
}

@book{harris_algebraic_1992,
	address = {New York},
	series = {Graduate {Texts} in {Mathematics}},
	title = {Algebraic {Geometry}: {A} {First} {Course}},
	volume = {133},
	shorttitle = {Algebraic {Geometry}},
	doi = {10.1007/978-1-4757-2189-8},
	publisher = {Springer},
	author = {Harris, Joe},
	month = sep,
	year = {1992},
}

@article{matsubara-heo_four_2023,
	title = {Four lectures on {Euler} integrals},
	volume = {75},
	issn = {2590-1990},
	doi = {10.21468/SciPostPhysLectNotes.75},
	journal = {SciPost Physics Lecture Notes},
	author = {Matsubara-Heo, Saiei-Jaeyeong and Mizera, Sebastian and Telen, Simon},
	month = oct,
	year = {2023},
	note = {\href{https://arxiv.org/abs/2306.13578v2}{ArXiv}},
}

@misc{pokidkin_combinatorics_2025,
	title = {Combinatorics behind discriminants of polynomial systems},
	doi = {10.48550/arXiv.2509.02963},
	publisher = {arXiv},
	author = {Pokidkin, Vladislav},
	month = sep,
	year = {2025},
	note = {\href{https://arxiv.org/abs/2509.02963v2}{ArXiv}},
}

@article{mizera_landau_2022,
	title = {Landau discriminants},
	volume = {2022},
	issn = {1029-8479},
	doi = {10.1007/JHEP08(2022)200},
	number = {8},
	journal = {Journal of High Energy Physics},
	author = {Mizera, Sebastian and Telen, Simon},
	month = aug,
	year = {2022},
	note = {\href{https://arxiv.org/abs/2109.08036v2}{ArXiv}},
	pages = {200},
}

@article{furukawa_combinatorial_2021,
	title = {A combinatorial description of dual defects of toric varieties},
	volume = {23},
	issn = {0219-1997, 1793-6683},
	doi = {10.1142/S0219199720500017},
	number = {01},
	journal = {Communications in Contemporary Mathematics},
	author = {Furukawa, Katsuhisa and Ito, Atsushi},
	month = feb,
	year = {2021},
	note = {\href{https://arxiv.org/abs/1605.05801v2}{ArXiv}},
	pages = {2050001},
}
	
\end{document}